%% file: main.tex
\documentclass[10 pt]{amsart}

\input{StandardPaper2.tex}

\pdfoutput=1
\usepackage{mathrsfs}
\usepackage{graphicx}
\DeclareMathAlphabet{\mathpzc}{OT1}{pzc}{m}{it}
\DeclareMathAlphabet{\mathantt}{OT1}{antt}{li}{it}
\graphicspath{ {images/} }
\DeclareMathOperator{\ad}{ad}
\newenvironment{hproof}{%
  \proof}{\endproof}
  \newcommand{\zf}{\mathpzc{f}}
\begin{document}
\title[Harnack inequalities for curvature flows]{Harnack inequalities for curvature flows in Riemannian and Lorentzian manifolds}
\author[P. Bryan]{Paul Bryan}
\address{School of Mathematics and Physics, The University of Queensland, St Lucia, Brisbane, 4072, Australia}
\email{pabryan@gmail.com}
\author[M.N. Ivaki]{Mohammad N. Ivaki}
\address{Institut f\"{u}r Diskrete Mathematik und Geometrie, Technische Universit\"{a}t Wien,
Wiedner Hauptstr. 8--10, 1040 Wien, Austria}
\address{Department of Mathematics and Statistics, Concordia University, Montreal, QC,
H3G 1M8, Canada}
\email{mohammad.ivaki@tuwien.ac.at}
\author[J. Scheuer]{Julian Scheuer}
\address{Albert-Ludwigs-Universit\"{a}t,
Mathematisches Institut, Eckerstr. 1, 79104
Freiburg, Germany}
\email{julian.scheuer@math.uni-freiburg.de}
\dedicatory{}
\subjclass[2010]{}
\date{\today}
\begin{abstract}
We obtain Harnack estimates for a class of curvature flows in Riemannian manifolds of constant non-negative sectional curvature as well as in the Lorentzian Minkowski and de Sitter spaces. Furthermore, we prove a Harnack estimate with a bonus term for mean curvature flow in locally symmetric Riemannian Einstein manifold of non-negative sectional curvature. Using a concept of ”duality” for strictly convex hypersurfaces, we
also obtain a new type of inequalities, so-called ”pseudo”-Harnack inequalities, for expanding flows in the sphere and in the hyperbolic space.
\end{abstract}

\maketitle
\tableofcontents
\section{Introduction}
\label{sec:intro}
Let $(N=N^{n+1},\bar{g})$, $n\geq 2$, be a Riemannian or Lorentzian manifold and let $M=M^{n}$ be a smooth, complete and orientable manifold. For flat ambient spaces, we use $\ip{\cdot}{\cdot}$ instead of $\bar{g}.$ Put $\s = 1$ in the Riemannian case and $\s = -1$ in the Lorentzian case.

Let
$x\cn M\x[0,T^{\ast})\ra N$
be a family of strictly convex\footnote{One of the choices of the normal $\nu$ yields a positive definite second fundamental form $h$ in the Gaussian formula \begin{align*}\bar\n_{x_{*}X}(x_{*}Y)=x_{*}\br{\n_{X}Y}-\s h(X,Y)\nu\quad\forall X,Y\in  T M. \end{align*}} and spacelike\footnote{The induced metric is positive definite.} embeddings, which evolves by the curvature flow
\eq{\label{Flow}\dot{x}=-\s f\nu-x_{\ast}(\operatorname{grad}_hf),}
where $\nu$ is a unit normal vector field along $M_{t}=x(M,t)$ (which satisfies $\s=\bar{g}(\nu,\nu)$ from the spacelike condition), and with $\operatorname{grad}_hf$ defined by
\eq{
 h(\grad_h f, X) = df (X)\quad\forall X\in TM
}
or in coordinates,
\[\operatorname{grad}_hf:=b^{ij}df(\partial_j)\partial_i.\]
Here $(b^{ij})$ is the inverse of the second fundamental form $(h_{ij})$. The speed $f$ is a smooth and strictly monotone function of the principal curvatures, which may also depend on other data, depending on the ambient space, compare \cref{SpeedAss}.

Let \(r:M\times [0,T^{\ast})\to M\) be the one-parameter family of diffeomorphisms associated with $\grad_h f$; that is, $r(\cdot,0)=\operatorname{id}$ and $\dot{r}=\operatorname{grad}_hf$. If $M$ is compact, then for each $t \in [0, T^{\ast}),$ $r_t : M \to M$ is uniquely defined and a diffeomorphism. If $M$ is co-compact so that $M/G$ is compact where $G$ is a Lie Group acting on $M$, and the flow $M_t$ is invariant under $G$, again $r_t$ is a uniquely defined diffeomorphism for each $t \in [0, T^{\ast})$. Defining $\check{x} (\xi, t) := x (r(\xi,t), t)$, $\check{\nu} (\xi, t) :=\nu (r(\xi,t), t)$ and $\check{f}(\xi,t):=f(r(\xi,t), t)$, we see that the flow \eqref{Flow} is equivalent to the curvature flow
\eq{\label{FlowStandard}
\check{x}&:M\times [0,T^{\ast})\to N\\
\dot{\check{x}} &= -\s \check{f} \check{\nu}.}
We call $\check{x}$ the standard parameterization as in \cite{Andrews:09/1994}.

Differential Harnack inequalities are pointwise derivative estimates which usually enable one to compare the speed of a solution to a curvature flow at different points in space-time.
Central to our approach in obtaining Harnack inequalities for a class of curvature flows (\ref{FlowStandard}) is a reparameterization of the flow given by the flow \eqref{Flow}.
In a Euclidean background, $N = \R^{n+1}$, the Gauss map $\nu: M\times [0,T^{\ast}) \to \S^n$ is a diffeomorphism for $t$ if $x(M,t)$ is strictly convex. The Gauss map parameterization $y: \S^n\times [0,T^{\ast}) \to \R^{n+1}$, cf.~\cite{Andrews:09/1994}, is such that $\nu(y(z,t),t) = z$ for all $z \in \S^n$ whence $\dot{\nu} = 0$. Furthermore, calculations may be performed with respect to the fixed, canonical, round metric $g_{\operatorname{can}}$ on $\S^{n}$. These two properties, a static metric and static normal provide immense benefit, not only in simplifying the generally long computations associated with differential Harnack inequalities, but also by lending insight into why such long computations yield such a simple, elegant differential Harnack inequality.

The Gauss map parameterization just described is manifestly Euclidean, and given the utility of such a parameterization, analogous results in other background spaces should be highly prized. The cornerstone of our approach is that the normal \(\nu\) is static in the parameterization \eqref{Flow} and the time derivative of the induced metric \(g\) is only felt through the changing parameterization, $x$. See (\ref{BE-n}), analogous to the Gauss map parameterization, valid in arbitrary backgrounds.

For the Harnack quantity we define
\begin{equation}
\label{eq:Q}
u := \frac{\dot{f}}{f}.
\end{equation}
Therefore, $u=\partial_t \ln |f|$ just as for Li-Yau \cite{LiYau:/1986} and Andrews \cite{Andrews:09/1994}. Then writing $u$ in the standard parameterization, we find that
\eq{\label{eq:change_param}
  \frac{\dot{f}}{f} &= \fr{\dot{\check{f}} - df (\dot{r})}{f}= \frac{\dot{\check{f}}-\check{b}(d\check{f},d\check{f})}{\check{f}},
}
which is precisely the standard Harnack quantity in the Euclidean space.

Our first theorem includes previously known Harnack inequalities in the Euclidean space and extends them by allowing the speed to depend on the ``support function". Furthermore, it provides Harnack inequalities for a class of curvature flows in the Minkowski space which are completely new.

Suppose there is a subgroup $G$ of future preserving isometries of the Minkowski space  such that $I(x(M))=x(M)$ for all $I\in G$ and $G$ acts properly discontinuously on $M.$ Let us put $K=M/G$. If $K$ is compact, we say that $M$ is co-compact. Let $I_{\ast}$ denote the linear part of $I\in G$ (e.q., $I=I_{\ast}+\vec{v}$ such that $I_{\ast}\in O^+(n,1),~\vec{v}\in\mathbb{R}^{n,1}$,  where $O^+(n,1)$ is the space of future-preserving linear transformations \emph{preserving the Lorentzian inner product} and $\mathbb{R}^{n,1}$ denotes the Minkowski space) and also put $G_{\ast}=\{I_{\ast}:I\in G\}.$ If in addition $G_{\ast}=G$, we say $M$ is standard.

Write $\mathbb{H}^n$ for the hyperbolic space. A function $\psi:\mathbb{H}^n\to \mathbb{R}$ is called $G_{\ast}$-invariant, if $\psi(I_{\ast}z)=\psi(z)$ for all $z\in \mathbb{H}^n$ and $I\in G.$ Therefore, $\psi:\mathbb{H}^n/G_{\ast}\to \mathbb{R}$ is well-defined.
\Theo{thm}{Euclidean}{
Let $N=N^{n+1}$ be either the Euclidean space $\R^{n+1}$ or the Minkowski space $\R^{n,1}$ and $M=M^{n}$ be a smooth, connected, complete and orientable manifold, which is compact in case $N=\R^{n+1}.$ Let
$x\cn M\x (0,T^{\ast})\ra N$
be a family of strictly convex, spacelike embeddings that solves the flow equation\footnote{For a better readability we omit the $\check{\hp{x}}$ in this and the following theorems for flows in the standard parametrization.}
\eq{\label{StandardFlow}\dot{x}=-\s f\nu,}
where $f=\p(s)\psi(\nu)\mrm{sgn}(p)F^{p}$ and
\begin{itemize}
\item $p\neq 0,$
\item  $s=\sigma\langle x,\nu\rangle$ is the support function,
\item  If $p\neq -1,$ $\p\in C^{\8}(\R_{+})$ is positive and satisfies
\eq{\s\p'\leq 0\quad\mbox{and}\quad \mrm{sgn}\br{p(p+1)}\br{\fr{1-p}{p}\p'^2+\p''\p}\geq 0,}
\item $F$ is a positive, strictly monotone, $1$-homogeneous curvature function that is inverse concave for $-1< p$ and inverse convex  for $p> -1$.
\end{itemize}
Suppose one of the following conditions holds:
\begin{enumerate}
  \item $N=\mathbb{R}^{n+1}$, $\psi\in C^{\8}(\S^{n})$ is positive and the solution is compact, strictly convex and if $\varphi\neq 1$ then $s(\cdot,t)>0$ for all $t.$
  \item $N=\mathbb{R}^{n,1}$, $\varphi=\psi\equiv1$, the solution is co-compact, spacelike and strictly convex.
  \item $N=\mathbb{R}^{n,1}$, $\varphi\equiv1$, the solution is co-compact, spacelike and strictly convex and $\psi:\mathbb{H}^n\to \mathbb{R}_+$ is a $G_{\ast}$-invariant, smooth function.
  \item $N=\mathbb{R}^{n,1}$, $\psi:\mathbb{H}^n\to \mathbb{R}_+$ is a $G_{\ast}$-invariant, smooth function, $s(\cdot,t)>0$ for all $t$, and the solution is standard, spacelike and strictly convex.
\end{enumerate}
Then for $p>-1$ the following Harnack inequality holds
\eq{\label{DiffHarnack}\del_{t}f-b(\nabla f,\nabla f)+\frac{p}{(p+1)t}f\geq 0\quad \forall t>0,}
and the inequality is reversed if $p<-1.$

Also, for $p=-1,\p=1$, under either of these four conditions the following statements hold:
\begin{enumerate}
  \item If $F$ is inverse concave,  then \eq{\inf \fr{\dot{f}-b(\n f,\n f)}{f}\quad \mbox{is increasing}}
  \item If $F$ is inverse convex, then \eq{\sup \fr{\dot{f}-b(\n f,\n f)}{f}\quad \mbox{is decreasing}.}
\end{enumerate}
}
\Theo{rem}{QuotientsofEuc}{
The proof of \cref{Euclidean} does not make any use of the simply connectedness of $\R^{n+1},$ so it is also possible to allow $N$ to be a quotient of $\R^{n+1},$ for example, a flat torus $\mathbb{T}^{n},$ $n\geq 3$.
}
\Theo{rem}{s well-defined on K}{Note for a standard, spacelike and strictly convex hypersurface $x(M)$, $s$ is well-defined on $K:$
\begin{align*}
s(Ix)&=-\langle Ix,\nu(Ix)\rangle=-\langle Ix,I_{\ast}\nu(x)\rangle=-\langle Ix,I\nu(x)\rangle=-\langle x,\nu(x)\rangle=s(x).
\end{align*}
}
\Theo{rem}{positive support function}{Let $\mathcal{U}$ denote the interior of $\{\langle z,z\rangle\leq 0, z_0\geq 0\}$. If $M$ is a standard, spacelike, strictly convex hypersurface that is contained in $\mathcal{U}$, then $M$ has a positive support function, cf. \cite[equ.~(13)]{FillastreVeronelli:02/2015}.
}
\Theo{rem}{expanding hyperbolic space}{Assume $N=\mathbb{R}^{n,1}$ and $F$ is a positive, $1$-homogeneous curvature function. Consider $x(M,t)=((1+p)f(1,\ldots,1)t)^{\frac{1}{1+p}}\mathbb{H}^{n}$ for $t\in(0,\infty)$, a solution to the expanding flow with $f=F^p$ (assuming $\varphi=\psi=1$). Then equality holds in the Harnack inequality. In fact, the support function of $x(M,t)$ is given by $s_t=((1+p)f(1,\ldots,1)t)^{\frac{1}{1+p}}.$ Hence, we have
\[\dot{s}_t=\frac{f(1,\ldots,1)}{((1+p)f(1,\ldots,1)t)^{\frac{p}{1+p}}}=f(x,t).\]
This verifies that $x(M,t)$ serves as a solution for any $t>0.$ Also, calculate
\begin{align*}
\partial_tf&=-\frac{pf(1,\ldots,1)^2}{((1+p)f(1,\ldots,1)t)^{\frac{1+2p}{1+p}}}\\
\frac{p}{(p+1)t}f&=\frac{p}{(p+1)t}\frac{f(1,\ldots,1)}{((1+p)f(1,\ldots,1)t)^{\frac{p}{1+p}}}.
\end{align*}
Therefore, for this particular solution the equality is obtained in the Harnack inequality. Note that if $t\to 0$, then $x(M,t)\to\{\langle z,z\rangle=0: z_0\geq 0\}$ (e.q., boundary of $\mathcal{U}$) with support function  equal to zero.
}
\cref{Euclidean} includes and extends (even in the Euclidean case) the previously known differential Harnack estimates in \cite{Andrews:09/1994}, \cite{Chow:06/1991,Li:/2011,Wang:11/2007}. For more general functions of the mean curvature in the Euclidean case see \cite{Smoczyk:/1997}. To our knowledge, the only available Harnack estimates for curvature flows having the support function in their speeds are for centro-affine normal flows \cite{Ivaki:11/2015,Ivaki:09/2015}. In this respect, our result is new even in the Euclidean case.

In other ambient spaces, far less is known, due to the complications which arise from the ambient curvature tensor.
So far, the only setting of non-constant sectional curvature for which we could obtain a Harnack inequality with a bonus term for the mean curvature flow is the locally symmetric Riemannian Einstein manifolds of non-negative sectional curvature.
\Theo{thm}{Einstein}{
Let $N=N^{n+1}$ be a locally symmetric Riemannian Einstein manifold of non-negative sectional curvature. Assume that $M=M^{n}$ is a smooth, connected, compact and orientable manifold. Then along any strictly convex solution $x\cn M\x (0,T^{\ast})\ra N$ to the mean curvature flow
\eq{\dot{x}=-H\nu} there holds
\[\del_{t}H-b(\nabla H,\nabla H)-\fr{\-R}{n+1}H+\fr{1}{2t}H\geq 0,\]
where $\-R$ is the constant scalar curvature of $N$.
}
Examples of suitable $N$ satisfying the assumptions of the theorem are irreducible symmetric spaces of compact type and quotients thereof \cite[7.75]{Besse:/1987}. In particular, an interesting example that satisfies the assumptions of this theorem is the complex projective space $N^{2n}=\mathbb{CP}^n.$ Compare \cite{PipoliSinestrari:04/2016} for a recent result on mean curvature flow in $\mathbb{CP}^n$.

If we have a more symmetric ambient space, we can obtain Harnack inequalities for a larger class of speeds. The next theorem includes our Harnack inequalities from \cite{BryanIvaki:08/2015,BIS1} and presents new Harnack inequalities, for example, in de Sitter space. Note as with \Cref{Euclidean} (see \Cref{QuotientsofEuc}) and \Cref{Einstein} the results hold for quotients and not just the simply connected case.
\Theo{thm}{SphereHarnack}{
Assume that $f =F^p$ with $0< p\leq 1,$ where $F$ is a positive, strictly monotone, convex, $1$-homogeneous curvature function. Let $M=M^{n}$ be a smooth, connected, compact, orientable manifold
and $x\cn M\x (0,T^{\ast})\ra N$ be a spacelike solution to the flow equation
\eq{\dot{x}=-\s f\nu.}
Suppose either
\begin{enumerate}
  \item $N$ is a Riemannian $(\s=1)$ spaceform with constant sectional curvature $K_{N}=1$, and the solution is strictly convex or
  \item $N$ is a Lorentzian $(\s=-1)$ spaceform with constant sectional curvature $K_{N}=1$, and the solution satisfies $0<\kappa_i\leq 1.$
\end{enumerate}
Then the following Harnack inequality holds along the flow:
\[
\partial_t f-b(\nabla f,\nabla f)+\frac{p}{(p+1)t}f\geq 0\quad \forall t>0.
\]
}

\begin{rem}
In \cref{Einstein,,SphereHarnack}, it would not effect the result, if we attached an anisotropic factor to the respective speeds $H$ and $F^p$, i.e., if we considered
\eq{f=\psi(\nu)H,\quad~\text{or}~f=\psi(\nu)F^p, }
where $\psi$ is a positive smooth function on the unit sphere bundle in $TN,$ which is invariant under parallel transport in $(N,\-g)$.
\end{rem}

Furthermore, employing  duality, we obtain ``pseudo"-Harnack inequalities for a class of curvature flows in the spherical and the hyperbolic space.
\Theo{thm}{pseudoHarnack}{
Suppose $F$ is a positive, strictly monotone, $1$-homogeneous, inverse convex curvature function and $f =-F^p$ with $-1\leq p<0$. Let $M=M^{n}$ be a smooth, connected, compact, orientable manifold and  $x\cn M\x (0,T^{\ast})\ra N$ be a solution to the flow equation
$\dot{x}=- f\nu.$
Suppose either
\begin{enumerate}
  \item $N$ is the sphere, and the solution is strictly convex or
  \item $N$ is the hyperbolic space, and the solution is strictly horoconvex \footnote{All principal curvatures are greater than $1$.}.
\end{enumerate}
Then the following inequality holds along the flow:
\[
\partial_t F^p+\frac{p}{(p-1)t}F^p\geq 0 \quad \forall t>0.
\]
}
The term pseudo-Harnack reflects the fact that the inequality in \Cref{pseudoHarnack} does not have the gradient term as opposed to the inequalities in Theorems \ref{Euclidean} and \ref{SphereHarnack} and thus would not allow one to compare the solution at different points in space-time, nevertheless, it is a point-wise estimate on $\partial_tf$, which is independent of the initial data. This new type of inequality suggests while in a negatively curved ambient space the standard Harnack quantity $u$ may fail to yield any interesting inequality, yet a weaker form (obtained by dropping the gradient term) may provide a useful inequality.

\subsubsection*{Connection to the cross curvature flow}
In \cite{ChowHamilton:/2004}, Chow and Hamilton introduced an interesting fully nonlinear heat flow for negatively (or positively) curved metrics on a 3-manifold, called the ``cross curvature flow'' (in short ``XCF"). This nonlinear curvature flow of metrics is dual to the Ricci flow in the following sense. The identity map from a Riemannian 3-manifold to itself, where the domain manifold has the cross curvature tensor as the metric (assuming the sectional curvature is either everywhere negative or everywhere positive), is harmonic, while the identity map from a Riemannian 3-manifold to itself, where the target manifold has the  Ricci curvature tensor as the metric (assuming the Ricci curvature is either everywhere negative or everywhere positive), is harmonic. Chow and Hamilton prove a monotonicity formula for XCF and give strong indications that the XCF should deform any negatively curved metric on a compact 3-manifold to a hyperbolic metric, modulo scaling. Also, they express strong hopes that the XCF should enjoy a Harnack inequality. Recently, it has appeared in \cite{AndrewsChenFangMcCoy:/2015} that if the universal cover of the initial 3-manifold is isometrically embeddable as a hypersurface in Minkowski 4-space (or Euclidean 4-space), then the Gauss curvature flow of the hypersurface yields the cross curvature flow of the induced metric. When, also, the manifold is closed, the global existence and convergence hold \cite{AndrewsChenFangMcCoy:/2015}. In that case, it is a corollary of Theorem \ref{Euclidean} that indeed a Harnack estimate for XCF exists; see inequality (\ref{cross harnack}).

\subsubsection*{Moser parabolic Harnack Inequality}
A differential Harnack inequality of the form \eqref{DiffHarnack} is related closely to the well-known Moser-type Harnack inequalities. Note that
the normal speed $f$ of the flow evolves by a parabolic equation, and as such a parabolic Harnack inequality as derived by Moser \cite{Moser:02/1964} is expected. As initially described by Li and Yau \cite{LiYau:/1986} and later adapted by Hamilton in the case of curvature flows \cite{Hamilton:/1986,Hamilton:/1993,Hamilton:/1995}, integrating along space-time paths yields the Moser parabolic Harnack inequality. In fact, the Li-Yau-Hamilton type differential inequality is equivalent to the Moser parabolic Harnack inequality, a fact not often expressed explicitly and described here by the next theorem. All the Harnack inequalities described in our main theorems are all of Li-Yau-Hamilton type \eqref{eq:liyauhamilton_harnack}, and hence we obtain a Moser parabolic Harnack inequality \eqref{eq:moser_harnack} in all those cases.
\begin{thm}
\label{thm:moser_liyauhamilton}Let $N=N^{n+1}$ be a semi-Riemannian manifold and $M=M^{n}$ be a smooth, connected, complete, orientable manifold. Suppose $x\cn M\x (0,T^{\ast}) \ra N$ is a family of strictly convex embeddings satisfying
\eq{\dot{x}=-\s f\nu,}
where $f\cn  M\x (0,T^{\ast})\ra N$ is a smooth and nowhere vanishing function. Assume that $q\in C^{0}((0,T))$. Then
\begin{equation}
\label{eq:liyauhamilton_harnack}
\begin{cases}
    \frac{\partial_tf - h(\grad_{h} f, \grad_{h} f)}{f} \geq -q(t), & \hbox{if}~f>0, \\
    \frac{\partial_tf - h(\grad_{h} f, \grad_{h} f)}{f} \leq -q(t), & \hbox{if}~f<0,
  \end{cases}
\end{equation}
if and only if for all $x_1, x_2 \in M$ and $t_2 > t_1 > 0$ there holds
\begin{equation}
\label{eq:moser_harnack}
f (x_1, t_1) \leq \frac{e^{Q(t_2)}}{e^{Q(t_1)}} e^{\Delta/4} f (x_2, t_2),
\end{equation}
where $Q(t_2) - Q(t_1) = \int_{t_1}^{t_2} q(t) dt$,
\[
\Delta = \left\{
           \begin{array}{ll}
             \inf_{\gamma} \int_{t_1}^{t_2} \frac{1}{f} h(\dot\gamma, \dot\gamma) dt, & \hbox{if}~ f>0, \\
             \sup_{\gamma} \int_{t_1}^{t_2} \frac{1}{f} h(\dot\gamma, \dot\gamma) dt, & \hbox{if}~ f<0,
           \end{array}
         \right.
\]
and the infimum and supremum are taken over all smooth curves $\gamma$ with $\gamma(t_i) = x_i$, $i = 1,2$.

In particular,
\[
\frac{\partial_tf - h(\grad_{h} f, \grad_{h} f)}{f} \geq - \frac{p}{p+1} \frac{1}{t}
\]
if and only if
\[
f (x_1, t_1) \leq \left(\frac{t_2}{t_1}\right)^{\fr{p}{p+1}} e^{\Delta/4} f (x_2, t_2).
\]
%With a \emph{bonus term} $K_0$,
%\[
%\frac{\partial_tf - h(\grad_{h} f, \grad_{h} f)}{f} \geq - \frac{p}{p+1} \frac{K_0 e^{K_0 t}}{e^{K_0 t} - 1}
%\]
%if and only if
%\[
%f (x_1, t_1) \leq \left(\frac{e^{K_0 t_2} - 1}{e^{K_0 t_1} - 1}\right)^{\fr{p}{p+1}} e^{\Delta/4} f (x_2, t_2).
%\]
\end{thm}
\begin{proof}
%The last two statements follow immediately from the first taking those particular choices of $q(t)$.
Now, let $X$ be an arbitrary tangent vector to $M$. Note that
\eq{
h\br{\grad_{h} f + \frac{1}{2}X, \grad_{h} f + \frac{1}{2}X}\geq 0.
}
Therefore,
\begin{equation}
\label{eq:CauchySchwarzPolarisation}
h(\grad_{h} f, X) + \frac{1}{4} h(X, X) \geq - h(\grad_{h} f, \grad_{h} f).
\end{equation}
with equality precisely when $X=-2\grad_{h} f.$

Hence the Li-Yau-Hamilton differential inequality \eqref{eq:liyauhamilton_harnack} holds if and only if,
\begin{equation}
\label{eq:standard_harnack_all_vectors}
\frac{\partial_t f + h(\grad_{h} f, X)+ \frac{1}{4} h(X, X)}{f} \geq - q(t)\quad \forall X\in TM,
\end{equation}
or with the opposite inequality in the case $f < 0$.

Next we show that the Moser parabolic Harnack inequality, \eqref{eq:moser_harnack} is equivalent to equation \eqref{eq:standard_harnack_all_vectors} by integrating along space-time paths. Let $x_1,x_2 \in M$ and $\gamma\cn [t_1, t_2] \to M$ be any curve connecting $x_1$ at time $t_1$ to $x_2$ at time $t_2$; that is, $\gamma(t_i) = x_i, i = 1,2$. Keeping in mind that $f$ is either strictly positive or strictly negative, we have
\[
\begin{split}
\ln |f| (x_2, t_2) - \ln |f| (x_1, t_1) &= \int_{t_1}^{t_2} \partial_t \left[\ln |f| (\gamma(t), t))\right] dt \\
&= \int_{t_1}^{t_2} \frac{\partial_t f}{f} + \frac{1}{f} \gamma'(f)  dt \\
&= \int_{t_1}^{t_2} \frac{\partial_t f + h (\grad_{h} f, \dot\gamma)}{f} dt.
\end{split}
\]
Taking exponentials,
\begin{equation}
\label{eq:exponential_harnack}
\frac{f (x_2, t_2)}{f (x_1, t_1)} = \exp \left(\int_{t_1}^{t_2} \frac{\partial_t f + h (\grad_{h} f, \dot\gamma)}{f} dt\right)
\end{equation}
where we may drop the absolute value on $|f|$ since both numerator and denominator have the same sign.

Assuming equation \eqref{eq:standard_harnack_all_vectors} holds, we have
\[
\exp\left(\int_{t_1}^{t_2} \frac{\partial_t f + h (\grad_{h} f, \dot\gamma)}{f} dt\right) \geq e^{Q(t_1) - Q(t_2)} \exp\left(- \frac{1}{4} \int_{t_1}^{t_2} \frac{1}{f} h(\dot\gamma, \dot\gamma) dt\right),
\]
for every $x_1, x_2$ and every $\gamma$ joining $x_1$ at $t_1$ to $x_2$ at $t_2$. The opposite inequality holds when $f < 0$. Then using equation \eqref{eq:exponential_harnack}, we obtain for $f>0,$
\begin{equation}\label{abc1}
\frac{f (x_2, t_2)}{f (x_1, t_1)} \geq e^{Q(t_1) - Q(t_2)} \exp\left(- \frac{1}{4} \int_{t_1}^{t_2} \frac{1}{f} h(\dot\gamma, \dot\gamma) dt\right)
\end{equation}
and for $f<0,$
\begin{align}\label{abc2}
\frac{f (x_2, t_2)}{f (x_1, t_1)} \leq e^{Q(t_1) - Q(t_2)} \exp\left(- \frac{1}{4} \int_{t_1}^{t_2} \frac{1}{f} h(\dot\gamma, \dot\gamma) dt\right).
\end{align}
In (\ref{abc1}), since the left hand side is independent of $\gamma$, we may take the supremum of the right hand side over all $\gamma$ joining $x_1$ at $t_1$ to $x_2$ at $t_2$ to obtain
\[
\frac{f (x_2, t_2)}{f (x_1, t_1)} \geq \sup_{\gamma} \left\{e^{Q(t_1) - Q(t_2)} \exp\left(- \frac{1}{4} \int_{t_1}^{t_2} \frac{1}{f} h(\dot\gamma, \dot\gamma) dt\right)\right\} = e^{Q(t_1) - Q(t_2)} e^{-\Delta}.
\]
Rearranging gives the Moser parabolic Harnack \eqref{eq:moser_harnack}. Similarly in (\ref{abc2}), take the infimum and rearrange to obtain the Moser parabolic Harnack \eqref{eq:moser_harnack}

Conversely, if the Moser parabolic Harnack \eqref{eq:moser_harnack} holds, then equation \eqref{eq:exponential_harnack} implies that
\[
\exp \left(\int_{t_1}^{t_2} \frac{\partial_t f + h (\grad_{h} f, \dot\gamma)}{f} dt\right) \geq \frac{e^{Q(t_1)}}{e^{Q(t_2)}} e^{-\Delta/4}
\]
with the opposite sign when $f < 0$. Taking logarithms yields
\[
\begin{split}
\int_{t_1}^{t_2} \frac{\partial_t f + h (\grad_{h} f, \dot\gamma)}{f} dt &\geq -\int_{t_1}^{t_2} q(t) dt - \Delta/4 \\
&\geq - \int_{t_1}^{t_2} q(t) dt -  \frac{1}{4} \int_{t_1}^{t_2} \frac{1}{f} h(\dot\gamma, \dot\gamma) dt
\end{split}
\]
for every $x_1, x_2$, $t_1, t_2$ and $\gamma$. Hence the inequality holds pointwise which is precisely equation \eqref{eq:standard_harnack_all_vectors}.
\end{proof}

\subsubsection*{Solitons}
The Harnack inequality is closely related to solitons in flat backgrounds (other backgrounds do not have sufficiently many isometries to provide symmetries of the flow).
The philosophy put forward by Hamilton in \cite{Hamilton:/1993,Hamilton:/1995} is that equality should be attained on expanding solitons, just as equality in the Li-Yau Harnack inequality \cite{LiYau:/1986} which is attained by the heat kernel, itself an expanding soliton. Thus Hamilton follows a procedure of differentiating the soliton equation to obtain soliton identities which eventually lead to the appropriate form for the Harnack quantity.

We follow this philosophy by showing that the parametrization \eqref{Flow} is naturally suited to the deduction of Harnack inequalities. For the purposes of this discussion it is enough to consider Euclidean space $(N, \bar{g}) = (\R^{n+1}, \ip{\cdot}{\cdot})$, homothetic solitons, and degree $p$-homogeneous speeds $f$, $p\neq -1$. Similar arguments also apply to Minkowski space.

Let $M$ be a smooth, connected, compact and orientable manifold. A homothetic soliton may be described as a pair $(x_0, \lambda)$ with $x_0\cn M \to N$ an immersion and $\lambda : [0, T^{\ast}) \to \R$ a smooth, positive function satisfying
\begin{equation}
\label{eq:soliton}
\begin{cases}
x(\xi, t) &= \lambda(t) x_0(\xi), \\
\lambda(0) &= 1, \\
\ip{\partial_t x}{\nu} &= -f(\cW).
\end{cases}
\end{equation}
Simple scaling arguments give
\begin{equation}
\label{eq:scaling}
\begin{cases}
\cW &= \frac{1}{\lambda} \cW_0, \\
\nu &= \nu_0, \\
f &= \frac{1}{\lambda^p} f_0.
\end{cases}
\end{equation}
Here we think of $f(\xi, t) = f(\cW(\xi, t))$ as a smooth function $M \to \R$ and likewise for $f_0(\xi)=f(\cW(\xi, 0))$. We also, by the usual abuse of notation, write $x$ for the position vector field in $\R^{n+1}$ at the point $x$.

Using equations \eqref{eq:soliton} and \eqref{eq:scaling}  we have
\[
f(\cW_0) = -(\partial_t \lambda) \lambda^p \ip{x_0}{\nu_0},
\]
which leads to
\begin{equation}
\label{eq:soliton_equations}
\begin{cases}
f(\cW_0) &= C_0 \ip{x_0}{\nu_0}, \\
\lambda(t) &= \sqrt[p+1]{1 - (p+1)C_0 t},
\end{cases}
\end{equation}
where $C_0$ is a constant. This equation is necessary and sufficient for homothetic solitons, completely characterizing them. From \eqref{eq:scaling} we see that the normal $\nu$ is fixed under the flow \eqref{eq:soliton} and hence necessarily $x$ must evolve by \eqref{Flow}. To see this, let the flow
\[\dot{x}=-\s f\nu-x_{\ast}V\]
have the property $\del_t\nu=0$ for some $V\in TM$. For all $X\in TM$ we have
\[0=\ip{\del_t\nu}{x_{\ast}X}=-\ip{\nu}{\-\n_X \dot{x}}=Xf-h(X,V),\]
which is only possible if $V=\grad_h f.$
This was already pointed out by Chow in \cite{Chow:06/1991}, whereas he did not use this flow to deduce the Harnack inequality. Due to this relation, the reparametrization \eqref{Flow} seems naturally suited to Harnack inequalities, since under a homothetic soliton the ratio of maximal to minimal curvature is in fact constant in time.

Let us investigate the behavior of our proposed Harnack quantity
\[u=\fr{\dot f}{f}\]
on a soliton. By \eqref{eq:scaling} we get
\[u=-p\fr{\dot\l}{\l}=pC_0\l^{-(p+1)}\]
and
\eq{\label{eq:SolitonODE}\dot{u}=-(p+1)pC_0\l^{-(p+2)}\dot{\l}=(p+1)C_0\l^{-(p+1)}u=\fr{p+1}{p}u^2, \quad
                u(0)=pC_0.}
Therefore, the soliton ODE, \eqref{eq:SolitonODE}, is very simple to deduce (note on a soliton $u$ is just a function of time in the parametrization \eqref{Flow}). Hence the hope that \eqref{Flow} might simplify the excruciating calculations in obtaining  Harnack inequalities is justified. Indeed, one of the major achievements of the present paper is our ability to deduce the evolution of \eqref{eq:Q} for strictly convex flows in {\it{any}} Riemannian or Lorentzian ambient space for a huge range of speed functions. This is quite a surprise, having in mind the tremendous computational effort in previous works.

\subsection*{Acknowledgment}
The work of the first author was supported  in part by the EPSRC on a Programme Grant entitled ``Singularities of Geometric Partial Differential Equations'' reference number EP/K00865X/1. The work of the second author was supported by Austrian Science Fund (FWF) Project
M1716-N25 and the European Research Council (ERC) Project 306445. The work of the third author has been funded by the "Deutsche Forschungsgemeinschaft" (DFG, German research foundation) within the research grant "Harnack inequalities for curvature flows and applications", grant number SCHE 1879/1-1.

\section{Background and Notation}
\label{sec:background}
\subsection*{Notation and Basic Definitions}
\label{subsec:bg_notation}
For a semi-Riemannian manifold $(M,g)$,  flat- and sharp-operators are defined as follows. For
$T\in T_{\xi}^{l,k}(M)$
let $T^{\flat}\in T_{\xi}^{l-1,k+1}(M)$ be defined by the requirement
\eq{T^{\flat}(X_{1},\dots, X_{k+1},Y^{1},\dots,Y^{l-1}):=g(T(X_{1},\dots,X_{k},Y^{1},\dots,Y^{l-1},\cdot),X_{k+1})}
for all $X_i\in T_{\xi}M$ and $Y^{k}\in T_{\xi}^{\ast}M.$
In coordinates, this reads
\eq{(T^{\flat})^{j_{1}\dots j_{l-1}}_{i_{1}\dots i_{k+1}}=g_{i_{k+1}j_{l}}T^{j_{1}\dots j_{l}}_{i_{1}\dots i_{k}},}
i.e., the $\flat$ operator always lowers the last index to the last slot.
We stipulate the sharp operator to reverse this transformation, i.e.,
\eq{\br{T^{\sharp}}^{\flat}=T;}
equivalently,
\eq{T(X_{1},\dots, X_{k},Y^{1},\dots,Y^{l})=g(T^{\sharp}(X_{1},\dots, X_{k-1},Y^{1},\dots,Y^{l},\cdot),X_{k}).}
If the metric is denoted by some other symbol, i.e., $\-g$, these operators will also be furnished accordingly, e.g., $\-\flat$. We will also use this notation even if $g$ happens to be negative definite.

For a spacelike embedding into a semi-Riemannian manifold $(N^{n+1},\-g)$,
\begin{align*}x\cn M^{n}\ra N^{n+1},\end{align*}
we let $g=x^{*}\-g$ be the induced metric and the second fundamental form is defined by the Gaussian formula for some given local normal field $\nu$,
\eq{\label{Gauss}\-\n_{x_{*}X}(x_{*}Y)=x_{*}\br{\n_{X}Y}-\s h(X,Y)\nu\quad\forall X,Y\in  T M,}
where $\-\n$ and $\n$ denote the Levi-Civita connections of $\-g$ and $g$ respectively. The Weingarten map is given by
\begin{align*}g(\mc{W}(X),Y)=h(X,Y).\end{align*}
From this and differentiating $0=\-g(\nu,{x_{*}Y})$, we obtain the Weingarten equation
\eq{\label{Weingarten}\-g(\-\n_{x_{*}X}\nu,x_{*}Y)=h(X,Y)\quad\forall X,Y\in  T M.}
Generally, geometric quantities of the ambient manifold are denoted by an overbar, e.g., our definition of the $(1,3)$ Riemannian curvature tensor of $\-g$ is given by
\eq{\label{Riem}\overline{\mrm{Rm}}(\-X,\-Y)\-Z=\-\n_{\-X}\-\n_{\-Y}\-Z-\-\n_{\-Y}\-\n_{\-X}\-Z-\-\n_{[\-X,\-Y]}\-Z}
and the $(0,4)$ version is
\begin{align*}\overline{\mrm{Rm}}^{\-\flat}(\-X,\-Y,\-Z,\-W)=\-g\br{\overline{\mrm{Rm}}(\-X,\-Y)\-Z,\-W},\end{align*}
where we suppress the $\-\flat$, if no ambiguities are possible.
Hence we have the Codazzi equation
\eq{\label{Codazzi}\br{\n_{Z}h}(X,Y)=\n h (X,Y,Z)=\n h(X,Z,Y)-\overline{\mrm{Rm}}(\nu,X,Y,Z).}
Note that
\[\nabla h (Z,X,Y)=g(\nabla_Y \mc{W}(X),Z).\]
Therefore, we may rewrite (\ref{Codazzi}) equivalently as follows
\begin{align*}\n_Y\mc{W}(X)=\n_X\mc{W}(Y)-\br{\overline{\mrm{Rm}}(X,Y)\nu}^{\top},\end{align*}
where $\hp{}^{\top}$ denotes the projection onto $TM$ and
 we stipulate that whenever we insert $X\in  T M$ into ambient tensors, we understand $X$ to be the push-forward $x_{*}X.$

For a bilinear form $B$, $B^t$ denotes its transpose,
\begin{align*}B^t(X,Y)=B(Y,X)\end{align*}
and $B_{\mrm{sym}}$ denotes its symmetrization,
\begin{align*}B_{\mrm{sym}}:=\fr 12(B+B^t).\end{align*}

%%%%%%%%%%%%%%%%%%%%%%%%%%%%%%%%%%%%%%%%%%%%%%%%%%%
%%%%% OLD CF SECTION
\subsection*{Speed Functions}
\label{subsec:bg_speed}
We introduce the form of the speeds $f$ we consider in \eqref{Flow}. First we revisit some of the theory of curvature functions.
\subsection*{Curvature functions}
It is well-known that a symmetric function (i.e., invariant under permutation of variables) $\Phi\in C^{\8}(\G)$ on an open and symmetric domain $\G\sub\R^{n}$ induces a function $F\in C^{\8}(\O)$ on an open subset of endomorphisms of an $n$-dimensional real vector space $E$, which are selfadjoint with respect to some fixed underlying scalar product $E$; see, for example, \cite{Andrews:/1994b,Andrews:/2007,CaffarelliNirenbergSpruck:12/1985,Gerhardt:/2006}. These approaches all suffer from the drawback that certain well-known formulas for derivatives of $F$ only hold in direction of selfadjoint operators. Since our reparametrization \eqref{Flow} produces several non-selfadjoint operators, we would like to have extended versions of these formulas. In this section, we collect some of the properties which hold whenever $F$ is defined on an open subset of the space of endomorphisms $\cL(E)$, e.g.,  the mean curvature $H(\cW)=\tr(\cW)$. The details can be found in \cite{Scheuer:03/2017}.
\subsubsection*{Symmetric functions and Operator functions}
{\defn{
Let $E$ be an $n$-dimensional real vector space and $\G\sub\R^{n}$ be an open and symmetric domain.
\begin{itemize}
\item[(i)]  $\cL(E)$ denotes the space of endomorphisms of $E$ and $\mc{D}_{\G}(E)\sub \cL(E)$ is the set of all diagonalizable endomorphisms with eigenvalues in $\G$.
\item[(ii)] On $\mc{D}_{\R^{n}}(E)$ we define the {\it{eigenvalue map}} $\mrm{EV}$ by
\[\mrm{EV}(\cW):=\kappa=(\k_{1},\dots,\k_{n}),\]
where $\kappa$ is the ordered $n$-tuple of eigenvalues of $\cW$ with $\k_{1}\leq\dots\leq \k_{n}$.
\item[(iii)] Let $\Phi\in C^{\8}(\G)$ be a symmetric function. $F$ is said to be an {\it{associated operator function}} of $\Phi$, if there exists an open set $\O\sub \cL(E)$, such that $F\in C^{\8}(\O)$ and
\begin{align*}F_{|\mc{D}_{\G}(E)\cap \Omega}=\Phi\circ\mrm{EV}_{|\mc{D}_{\G}(E)\cap \Omega}.\end{align*}
\end{itemize}
 }}

It is convenient to give some examples right away.
\begin{example}
The power sums for $0\leq k\in \Z$ is defined by
\eq{p_{k}(\k):=\sum_{i=1}^{n}\k_{i}^{k}.}
The associated operator functions $P_k$ defined on $\O=\cL(E)$ are given by
\eq{P_k(\cW):=\mrm{Tr}(\cW^k).}
Write $s_k$ for the $k$-th elementary symmetric polynomial defined on $\G=\R^{n},$
\begin{align*}s_{k}(\k_{1},\ldots,\k_{n}):=\sum_{1\leq i_{1}<\dots<i_{k}\leq n}\prod_{j=1}^{k}\k_{i_{j}}.\end{align*}
It is well-known that $s_k$ can be written as a function of the power sums,
\eq{s_k=\chi(p_1,\ldots,p_m),}
where $\chi$ is a polynomial, cf. \cite{Mead:10/1992}. Hence the associated operator functions $H_k$ are
\eq{H_k=\chi(P_1,\ldots,P_m).}
Note that the following identity holds
\eq{H_k(\cW)=\fr{1}{k!}\fr{d^k}{d t^k}\det(I+t\cW)_{|t=0},}
cf. \cite[equ.~(2.1.31)]{Gerhardt:/2006}.

Moreover, if $\Phi\in C^{\8}(\G)$ for an open, symmetric domain $\G\sub\R^{n}$, then $\Phi$ can be written as a smooth function of the elementary symmetric polynomials (see \cite{Glaeser:01/1963}), and hence $\Phi$ can also be written as a smooth function of the power sums,
\eq{\label{SymPhi}\Phi=\rho(p_{1},\ldots,p_{m}).}
 and the associated operator function, defined on some open set $\O\sub\cL(E)$, is
\eq{\label{SymF}F=\rho(P_{1},\ldots,P_{m}).}
\end{example}

For such a pair of symmetric functions $\Phi\in C^{\8}(\G)$ and $F\in C^{\8}(\O),$ we will now state some of the properties of their derivatives, which in particular recover the well-known formulas when restricting these maps to selfadjoint transformations. However, note the difference in \eqref{DerSymmFct-A} with, for example, \cite[Thm.~5.1]{Andrews:/2007} and \cite[Lemma~2.1.14]{Gerhardt:/2006}.

{\thm{\label{DerSymmFct} The following two statements hold.
\begin{itemize}
\item[(i)] Let $E$ be an $n$-dimensional vector space, and assume that $\G\sub\R^{n}$ is open and symmetric. Consider $\Phi\in C^{\8}(\G)$ defined in \eqref{SymPhi} and let $F\in C^{\8}(\O)$ be the associated operator function \eqref{SymF}. Then we have
\begin{align*}
dF(\cW)(B)=\tr(F'(\cW)\circ B)\quad\forall\cW\in \O,~\forall B\in \cL(E)
\end{align*}
for some $F'(\cW)\in \cL(E).$ Moreover, if $\cW\in \mc{D}_{\G}(E)$, then $F'(\cW)$ and $\mc{W}$ are simultaneously diagonalizable. For a basis $\{e_1,\dots,e_n\}$ of eigenvectors for $\cW$ with eigenvalues $\kappa=(\k_1,\dots,\k_n),$ the eigenvalue $F^i$ of $F'(\cW)$ with eigenvector $e_i$ is given by
\begin{align*}
F^{i}(\cW)=\fr{\del \Phi}{\del\k_{i}}(\k).
\end{align*}
\item[(ii)] Suppose in addition that $\G$ is  convex, $\cW\in \mc{D}_{\G}(E)$ and  $(\eta^{i}_{j})$ is the matrix representation of some $\eta\in \cL(E)$ with respect to a basis of eigenvectors of $\cW$. Then there holds
\eq{\label{DerSymmFct-A}d^{2}F(\cW)(\eta,\eta)=\sum_{i,j=1}^n\fr{\del^2\Phi}{\del\k_i\del\k_j}\eta^i_i\eta^j_j+\sum_{i\neq j}^{n}\fr{\fr{\del \Phi}{\del \k_{i}}-\fr{\del \Phi}{\del\k_{j}}}{\k_{i}-\k_{j}}\eta^{i}_{j}\eta^{j}_{i},}
where $f$ is evaluated at $\k$.
The latter quotient is also well-defined in case $\k_{i}=\k_{j}$ for some $i\neq j$.
\end{itemize}
}
}
\begin{hproof}
By a direct calculation one can show that this result holds for all power sums and the chain rule carries this over to all functions of them.
Details can be found in \cite{Scheuer:03/2017}.
\end{hproof}
We will also need an associated map defined on bilinear forms with the function $F$. Let us write $\mc{B}(E)$ and $\mc{B}_{+}(E)$ for the set of bilinear forms and positive definite bilinear forms on $E$ respectively.
{\prop{\label{BilF}
Let $E$ be an $n$-dimensional real vector space, $\O\sub\cL(E)$ open and $F\in C^{\8}(\O)$ as in \eqref{SymF}. Define the open set
\eq{\hat{\O}:=\{(g,h)\in \mc{B}_{+}(E)\x\mc{B}(E)\cn h_{\mrm{sym}}^{\sharp_{g}}\in \O \}}
and a map
\begin{align*}
\cF&\cn \hat\O\ra \R\\
\cF(g,h)&:= F( h_{\mrm{sym}}^{\sharp_g}).
\end{align*}
Then $\cF$ is smooth and for any $a\in \mc{B}(E)$ we have
\eq{\label{BilF-A}d_{h}\cF(g,h)(a):=\fr{\del \cF}{\del h}(g,h)(a)=\tr(F'( h_{\mrm{sym}}^{\sharp_{g}})\circ a_{\mrm{sym}}^{\sharp_{g}})=dF(  h_{\mrm{sym}}^{\sharp_g}) (a_\mrm{sym}^{\sharp_{g}}).}
}}
\subsubsection*{Properties of symmetric functions}
Let us put
\begin{align*}\G_{+}:=\{\k\in\R^{n}\cn \k_{i}>0,~ i=1,\ldots,n\}.
\end{align*}
\Theo{defn}{InverseCF}{
Let $\Phi\in C^{\8}(\G_{+})$ be symmetric and assume $F\in C^{\8}(\O)$ is the associated operator function given by \eqref{SymF}. The {\it{inverse symmetric function}} of $\Phi$ is defined by
\eq{\~\Phi(\k_{i}):=\fr{1}{\Phi(\k_{i}^{-1})}}
and the associated operator function is defined as
\eq{\~F(\cW):=\fr{1}{F(\cW^{-1})}}
for all $\cW\in  \mrm{GL}_{n}(\R)$ with $\cW^{-1}\in \O$.
}
\begin{defn}\label{PropSymm}
Let $\G\sub\R^{n}$ be open and symmetric and $\Phi\in C^{\8}(\G)$ be symmetric.
\begin{itemize}
\item[(i)]~$\Phi$ is {\it{strictly monotone}}, if
\begin{align*}\fr{\del \Phi}{\del\k_{i}}(\k)> 0\quad\forall \k\in\G \quad \forall 1\leq i\leq n.\end{align*}
\item[(ii)] Assume in addition that $\G$ is a cone. $\Phi$ is {\it{homogeneous of degree $p\in \R$}}, if
\begin{align*}\Phi(\l \k)=\l^{p}\Phi(\k)\quad\forall \l>0\quad\forall \k\in\G.\end{align*}
\item[(iii)] $\Phi$ is {\it{inverse concave (inverse convex)}}, if $\~\Phi$
is concave (convex).
\end{itemize}
\end{defn}
These properties carry over to the associated operator function:
{\prop{\label{MonoF}
Let $E$ be an $n$-dimensional vector space, and assume that $\G\sub\R^{n}$ is open and symmetric. Consider $\Phi\in C^{\8}(\G)$ and $F\in C^{\8}(\O)$ as in \eqref{SymPhi} and \eqref{SymF}.
Then the following statements hold.
\begin{itemize}
\item[(i)] If $\Phi$ is strictly monotone, then $F'(\cW)$ has positive eigenvalues at every $\cW\in \mc{D}_{\G}(E)$ and the bilinear form $d_{h}\cF(g,h)$ from \cref{BilF} is positive definite at all pairs $(g,h)$ with $h_{\mrm{sym}}^{\sharp_g}\in \mc{D}_{\G}(E).$
\item[(ii)] If $\G$ is a cone and $\Phi$ is homogeneous of degree $p$, then $\mc{D}_{\G}(E)$ is a cone and $F_{|\mc{D}_{\G}(E)}$ is homogeneous of degree $p$.
%\item[(iii)]  If $\G$ is convex and $\Phi$ is concave, then $F$ satisfies
%\eq{\label{MonoF-A} d^{2}F(A)(\eta,\eta)\leq 0}
%for all $\eta$ with symmetric matrix representation with respect to a basis of eigenvectors of $A$. The reverse inequality holds if $f$ is convex.
 \end{itemize}
}}
%\begin{rem}The restricted concavity (convexity) in \eqref{MonoF-A} is due to the form of the second term in \eqref{DerSymmFct-A} and the fact that the trace of the square of a non-symmetric matrix may be negative. This explains why \eqref{MonoF-A} is only valid in selfadjoint directions.
%\end{rem}
\begin{rem}
Slightly abusing terminology, especially when it comes to convexity or concavity, we say $F$ is {\it{strictly monotone, homogeneous, concave or convex}}, if $\Phi$ has the corresponding properties.
\end{rem}

The following inequality for $1$-homogeneous curvature functions is very useful. The idea comes from \cite[Thm.~2.3]{Andrews:/2007} and also appeared in a similar form in \cite[Lem.~14]{BIS1}. The proof can be found in \cite{Scheuer:03/2017}.
{\prop{\label{InvConc}
Let $E$ be an $n$-dimensional real vector space. Let $\Phi\in C^{\8}(\G_{+})$ and $F\in C^{\8}(\O)$ be as in \eqref{SymPhi} and \eqref{SymF} with the further assumptions that $F$ is symmetric, positive, strictly monotone and homogeneous of degree one.
The following statement holds:

For every pair $\cW\in \mc{D}_{\G_{+}}(E)$ and $g\in \mc{B}_{+}(E)$ such that $\cW$ is selfadjoint with respect to $g$, we have
\eq{\label{InvConc-A}dF(\cW)(\ad_g(\eta)\circ \cW^{-1}\circ\eta)\geq F^{-1}\br{dF(\cW)(\eta)}^{2}\quad \forall \eta\in \cL(E),}
where $\ad_g(\eta)$ is the adjoint with respect to $g$.
%\item[(ii)] If $\Phi$ is inverse concave, then for every pair $A\in \mc{D}_{\G_{+}}(V)$ and $g\in \mc{B}_{+}(E)$ such that $A$ is selfadjoint with respect to $g$,
%\eq{\label{InvConc-B}d^{2}F(A)(\eta,\eta)+2dF(A)(\eta\circ A^{-1}\circ\eta)\geq 2F^{-1}\br{dF(A)\eta}^{2},}
%for all $g$-self-adjoint $\eta$.
}}
\begin{rem}\label{prop2.4rem}
A simple calculation reveals that if we set
\begin{align*}f=\mrm{sgn}(p)F^{p},\quad p\neq 0,\end{align*}
then from the inequality \eqref{InvConc-A} %\eqref{InvConc-B}
we have
\eq{\label{InvConc-A1}df(\cW)(\ad_{g}(\eta)\circ \cW^{-1}\circ\eta)\geq \fr{1}{p}f^{-1}\br{df(\cW)(\eta)}^{2}.}
%and
%\eq{\label{InvConc-B1}d^{2}f(A)(\eta,\eta)+2df(A)(\eta\circ A^{-1}\circ\eta)\geq \fr{p+1}{p}f^{-1}\br{df(A)\eta}^{2}}
%for all $g$-self-adjoint $\eta$.
%Also, if $\Phi$ is convex then
%\begin{align*}d^2f(A)(\eta,\eta)\geq\fr{p-1}{p}f^{-1}\br{df(A)(\eta)}^2\end{align*}
%for all $g$-self-adjoint $\eta$.
\end{rem}
{\example{
Let us define
\[\Gamma_k:=\{\k\in \R^{n}:~ s_1(\k)>0,\ldots, s_{k}(\k)>0\}.\]
\begin{enumerate}
  \item $s_{1}(\k)=H_{1}(\mc{W})=\mrm{Tr}(\mc{W})$
  is strictly monotone and inverse concave on $\G_{1}.$
    \item
$s_{n}(\k)=H_{n}(\mc{W})=\det(\mc{W})$
is strictly monotone on $\G_{n}.$
\item The quotients
\begin{align*}
q_k&:\Gamma_k\to \R,\quad \kappa\mapsto\fr{s_k}{s_{k-1}}
\end{align*}
are monotone and concave; see \cite[Cor.~5.3]{Andrews:/2007} and \cite[Thm.~2.5]{HuiskenSinestrari:09/1999}. Moreover, they are inverse concave; cf. \cite[Cor.~2.4, Thm.~2.6]{Andrews:/2007}.
\end{enumerate}
}}
\subsubsection*{Curvature functions}
%It is easy to transfer these previous notions to the bundle $T^{1,1}(M)$ of a smooth manifold $M$.
\begin{defn}\label{DefCF}
Let $M$ be a smooth manifold and $\O\sub T^{1,1}(M)$ be an open set. A function $F\in C^{\8}(\O)$ is said to be  a {\it{curvature function}} if there is a symmetric function $\Phi\in C^{\8}(\G)$ of the form \eqref{SymPhi} on an open and symmetric set $\G$ with the following property: For each $\xi\in M$, the function $F$ restricted to the fiber at $\xi$ is the associated operator function of $\Phi$ given by \eqref{SymF}.
\end{defn}
  A curvature function $F$ is said to have the properties from \cref{PropSymm}, if $\Phi$ has the corresponding properties.
%Also note that a curvature function can be written as a function of power traces: For $X\in \O\sub T^{1,1}(M)$, i.e., $X=(\xi,\cW)$ with $\cW\in T^{1,1}_{\xi}(M)$, we have
%  \eq{F(\xi,\cW)=\rho(\tr(\cW),\dots,\tr(\cW^m)),}
%  since $\Phi$ does not vary with the point $\xi$,
%  compare \eqref{SymF}.

The normal variation speeds for the flow (\ref{Flow}) do not solely depend on the principal curvatures, and they are of a more general form satisfying the following assumptions.
\Theo{ass}{SpeedAss}{
 $f$ is a non-vanishing velocity of the form
\eq{\label{Speed}f\cn \R_{+}\x \mathbb{U} &\x \O\ra \R\\ f(s,\nu,\cW)&=\mrm{sgn}(p)\p(s)\psi(\nu)F^{p}(\cW),}
where
\begin{itemize}
\item[(i)] $p\neq 0$,
\item[(ii)] $\mathbb{U}$ is the unit sphere bundle on $N$ (including timelike unit vectors),
\item[(iii)] $\p\in C^{\8}(\R_{+})$ is a positive function acting on the support function $s$ and $\p\equiv 1$ if $N$ is neither the Euclidean nor the Minkowski space,
\item[(iv)] $\psi\in C^{\8}(\mathbb{U})$ is a positive function on the unit bundle, such that $\psi$ is invariant under parallel transport in $(N,\-g)$,
\item[(v)] $F$ is a positive, strictly monotone and $1$-homogeneous curvature function of the form \eqref{SymF}, associated with a $\Phi\in C^{\8}(\G_{+})$, which is
\begin{itemize}
\item[(v-1)] inverse concave for $p>-1$ and inverse convex for $p<-1$, if $N=\R^{n+1}$ or $N=\R^{n,1}$,
\item[(v-2)] convex, if $N$ has constant nonzero sectional curvature and
\item[(v-3)] the mean curvature $H$, if $N$ has nonconstant sectional curvature.
\end{itemize}
\end{itemize}
}
{\rem{ The following remarks are in order:
\begin{itemize}

\item[(i)] Let us write $\pr\cn T^{1,1}(M)\ra M$ for the canonical projection. For every $X\in T^{1,1}(M)$, and $(v,0)\in T^{1,0}_{\pr(X)}(M)\x T^{1,1}_{\pr(X)}(M)\simeq T_{X}(T^{1,1}(M))$, there holds
\eq{\label{DefCF-1}dF(X)(v,0)=0.}
\pf{For any $v \in T^{1,0}(M)$, choose a curve $\a$ with $\a(0) = \pr(v)$, $\a'(0) = v$.
Let $Z\cn M \to T^{1,1}(M)$ be the zero section.
Then $t \mapsto Z(\alpha(t))$ is a curve in $T^{1,1}(M)$ with $(Z \circ \alpha)'(0) = (v, 0)$, hence $dF (v, 0) = \partial_t F(Z(\alpha(t)))|_{t=0}$.
Since $F$ is a curvature function, there holds $F(Z(\alpha(t))) \equiv \Phi(0)$, where the $0$ on the right-hand side is just the zero operator in fibre.
Thus $\partial_t F(Z(\alpha(t))) \equiv 0$ and so $dF(v, 0) = 0$ as required.}

\item[(ii)] In a local coordinate system, any point in $T^{1,1}(M)$ can be expressed as $(\xi^k,a^i_j)$, such that $(\xi^k)$ is a local coordinate system for $M$ and $(a^i_j)$ are the components of an arbitrary tensor field. For a curvature function $F$, due to \eqref{DefCF-1}, $dF$ acts only in the fibres; that is, $d_{\xi}F=0$. Hence by \cref{DerSymmFct} there exists an operator $F'\cn \O\ra T^{1,1}(M)$, such that for any $(1,1)$-tensor field $\mc{W}$ (which is a section of $T^{1,1}(M)$) and all $B\in T_\xi^{1,1}(M),~ v\in T^{1,0}_\xi(M)$, we have
\eq{dF(\xi,\cW(\xi))(v,B)&=d_{\cW}F(\xi,\cW(\xi))(B)\\
            &=\mrm{Tr}(F'(\xi,\cW(\xi))\circ B).}
For any vector field $X$ on $M$, metric $g$ on $M$, curvature function $F$ and any $g$-selfadjoint  $(1,1)$-tensor field $\cW$, we also have (and will frequently use)
\eq{X(F(\xi,\cW(\xi))&=d_{\cW}F(\xi,\cW(\xi))(\n_X \cW(\xi))\\
&=\mrm{Tr}\br{F'(\xi,\cW(\xi))\circ \n_X\cW(\xi)}\\
&=d_h\cF(g,h)\br{\n_X h},}
where $h=\cW^{\flat}$ and we assume $\n g=0$. A similar formula applies for time derivatives (where of course $\dot{g}$ is not zero in general).
However, note that on frequent occasions we will suppress the argument $\cW$ from $d_{\cW}F$, since it will be apparent from the subscript $\cW$ anyway.
\item[(iii)] For the more general speed function $f=\mrm{sgn}(p)\p(s)\psi(\nu)F^{p}$ we will write
\eq{\zf:=\mrm{sgn}(p)\p(s)\psi(\nu)\cF^{p}}
and
\eq{f'=|p|\p(s)\psi(\nu)F^{p-1}F',}
i.e., there holds
\eq{df(\xi,\cW)(v,B)=d_{\cW}f(\xi,\cW)(B)&=|p|\p(s)\psi(\nu)F^{p-1}\tr(F'(\xi,\cW)\circ B)\\
                    &=\tr(f'(\xi,\cW)\circ B).}
\end{itemize}
}}
%%%%%%%%%%%%%%%%%%%%%%%%%%%%%%%%%%%%%%%%%%%%%%%%%%
%%%%%%%%%%%%%%%%%%END OF OLD CF SECTION %%%%%%%%%

\section{Evolution equations}
We begin by collecting some basic evolution equations. The final aim is to deduce the evolution equation for the function
\begin{align*}u:=\fr{\dot{f}}{f}\end{align*}
under the flow \eqref{Flow},
\eq{\label{Flow2}\dot{x}=-\s f\nu-x_{*}V,}
where $f$ satisfies \cref{SpeedAss} and
\eq{\label{EvEq-1}V:=\grad_{h}f}
is the spatial gradient of $f$ with respect to the second fundamental form:
\eq{h(V,X)=Xf\quad\forall X\in  T M.}
Note that
\eq{\label{EvEq-2}\-\n_{X}\dot{x}&=-\s\-\n_{X}\br{f\nu}-\-\n_{X}V\\
                &=-\s h(V,X) \nu-\s f\-\n_{X}\nu-x_{*}\n_{X}V+\s h(X,V)\nu}
is tangential and hence we may define an endomorphism $A\in T^{1,1}(M)$ by
\eq{\label{A} x_{*}(A(X))=-\-\n_{X}\dot{x}.}
Since we are dealing with strictly convex hypersurfaces, the tensor
\begin{align*}\~g:=\fr{h}{f}\end{align*}
defines a symmetric non-degenerate bilinear form.
We also define a bilinear form associated with $A$:
\begin{align*}B(X,Y):=\~g(X,A(Y)).\end{align*} Note that
\begin{align*}V=\grad_{\~g}\ln |f|.\end{align*}
Also let us define
\begin{align*}\L(X,Y):=\overline{\mrm{Rm}}(\dot{x},x_{\ast}X,\nu,x_{\ast}Y).\end{align*}
Note that $B$ and $\L$ are generally \emph{not} symmetric.%\footnote{This is one of the reasons that convinced us to do the calculations invariantly. The other reason was Paul's persistence.}
\Theo{lemma}{SymmetriesB}{
There holds
\begin{align*}B(X,Y)&=B(Y,X)+\fr{1}{f}\overline{\mrm{Rm}}(\dot{x},\nu,X,Y)\\
        &=B(Y,X)+\fr 1f \L(X,Y)-\fr 1f\L(Y,X).\end{align*}
}
\pf{
For $X,Y\in  T M,$ due to the Weingarten equation \eqref{Weingarten} and \eqref{EvEq-2},
\begin{align*}fB(X,Y)=h(X,A(Y))&=\-g(\-\n_{X}\nu,A(Y))\\
                        &=\s f\-g(\-\n_{Y}\nu,\-\n_{X}\nu)+\-g(\-\n_{X}\nu,\n_{Y}V)\\
                    &=\s f\-g(\-\n_{X}\nu,\-\n_{Y}\nu)+h(X,\n_{Y}V).\end{align*}
Moreover, we use the Codazzi equation \eqref{Codazzi} to obtain
\begin{align*}h(X,\n_{Y}V)&=YXf-h(\n_{Y}X,V)-\n h(V,X,Y)\\
            &=XYf-h(\n_{X}Y,V)-\n h(V,Y,X)+\overline{\mrm{Rm}}(\nu,V,X,Y)\\
            &=h(Y,\n_{X}V)+\overline{\mrm{Rm}}(\nu,V,X,Y).\end{align*}
Hence the claim follows from the first Bianchi identity.
}
\subsection*{Basic evolution equations}
\Theo{lemma}{BasicEv}{
Along the flow \eqref{Flow} there hold
\eq{\label{BE-g}\dot{g}=-2A^{\flat}_{\mrm{sym}},}
\eq{\label{BE-n}\fr{\-\n}{dt}\nu =0,}
\eq{\label{BE-W}\dot{\mc{W}}=A\circ \mc{W}+\L^{\sharp},}
\eq{\label{BE-h}\dot{h}=-fB+\L,}
\eq{\label{BE-tildeg}\dot{\~g}=-B-\~gu+\fr{\L}{f},}
\eq{\label{BE-gradf}\dot{V}&=\grad_{\tilde{g}}u+A(V)+uV-\fr{1}{f}(\L^{t})^{\~\sharp}V,}
\eq{\label{BE-speed} \fr{\-\n}{dt}\dot{x}=u\dot{x}-x_{\ast}(\grad_{\tilde{g}}u)+\fr 1f x_{\ast}((\L^{t})^{\~\sharp}V ).}
}
\pf{Let $X,Y\in T M.$\newline
``\eqref{BE-g}'': By \eqref{EvEq-2} we have
\eq{\del_{t}\br{x^{*}\-g}(X,Y)&=\del_{t}\br{\-g(x_{*}X,x_{*}Y)}\\
                    &=\-g(\-\n_{X}\dot{x},x_{*}Y)+\-g(x_{*}X,\-\n_{Y}\dot{x})\\
                    &=-\-g(x_{*}(A(X)),x_{*}Y)-\-g(x_{*}(A(Y)),x_{*}X)\\
                    &=-g(A(X),Y)-g(X,A(Y)).}
``\eqref{BE-n}'': We have $0=\del_{t}\bar{g}(\nu,\nu)$. Since $0=\del_{t}\bar{g}(\nu,x_{*}X)$, we get
\eq{\bar{g}(\-\n_{\dot{x}} \nu,X)=-\bar{g}(\nu,\-\n_{\dot{x}}X)=-\bar{g}(\nu,\-\n_{X}\dot{x})=0.}
``\eqref{BE-W}'': Recall that \eqref{Weingarten} implies
\eq{\-\n_{x_{*}X}\nu=x_{*}\mc{W}(X).} Taking $\-\n_{\dot{x}}$, using (\ref{BE-n}) and \eqref{Riem} we calculate
\eq{\overline{\mrm{Rm}}(\dot{x},X)\nu=\-\n_{x_{*}\mc{W}(X)}\dot{x}+[\dot{x},x_{*}\mc{W}(X)]=-x_{*}(A(\mc{W}(X)))+x_{*}\dot{\mc{W}}(X).}
``\eqref{BE-h}'': Differentiate the Weingarten \eqref{Weingarten} equation to obtain
\begin{align*}\del_{t}h(X,Y)=&\-g\br{\-\n_{\dot{x}}\-\n_{X}\nu,Y}+\-g\br{\-\n_{X}\nu,\-\n_{\dot{x}}Y}\\            =&\overline{\mrm{Rm}}\br{\dot{x},X,\nu,Y}-h(X,A(Y))\\
=&\L(X,Y)-fB(X,Y).
\end{align*}
``\eqref{BE-tildeg}'': It follows directly from \eqref{BE-h}.\newline
``\eqref{BE-gradf}'':
\begin{align*}Xu&=X\del_{t}\ln|f|\\    &=\del_{t}\br{\~g(X,\grad_{\tilde{g}}\ln|f|)}=\dot{\~g}(X,V)+\~g(X,\dot{V})\\
    &=-B(X,V)+\fr{1}{f}\L(X,V)-\~g(X,V)u+\~g(X,\dot{V})\\
    &=-\~g(X,A(V))+\fr{1}{f}\L(X,V)-\~g(X,V)u+\~g(X,\dot{V}).
\end{align*}
``\eqref{BE-speed}'':
\begin{align*}\fr{\-\n}{dt}\dot{x}=&-\s\dot{f}\nu-\fr{\-\n}{dt}(x_{\ast}V)\\
                =&-\s\dot{f}\nu-\-\n_{x_{\ast}V}\dot{x}-[\dot{x},x_{\ast}V]\\
                =&-\s\dot{f}\nu+x_{\ast}(A(V))-x_{\ast}\dot{V}\\
                =&-\s uf\nu+x_{\ast}(A(V))-x_{\ast}(\grad_{\tilde{g}}u)-x_{\ast}(A(V))-ux_{\ast}V+x_{\ast}(\fr{1}{f}(\L^{t})^{\~\sharp}V)\\
                =&u\dot{x}-x_{\ast}(\grad_{\tilde{g}}u)+\fr 1f x_{\ast}((\L^{t})^{\~\sharp}V)
.\end{align*}}
\subsection*{Evolution equations involving the affine connection}
From now on, to simplify the calculations, we will work with the {\it{affine connection}} $\~\n$ induced by the transversal vector field $\dot{x}.$

For $X,Y\in T M$ we have a decomposition given by\footnote{In fact, $\hat\n_{X}Y=\n_{X}Y+\~g(X,Y)V$ and hence $\hat{\n}$ is torsion free.}
\begin{align*}\-\n_{X}Y=x_{\ast}(\hat\n_{X}Y)+\~g(X,Y)\dot{x}.\end{align*}
However, $\hat{\n}$ is not the Levi-Civita connection for the so-called {\it{affine fundamental form}} $\~g.$ Let $\~{\n}$ denote the Levi-Civita connection of $\~g$ and define the difference tensor $D$ of type $(1,2)$ by
\begin{align*}D_{X}Y:=\hat{\n}_{X}Y-\~{\n}_{X}Y.\end{align*}
Since both $\tilde{\nabla}$ and $\hat{\nabla}$ are torsion free, we have $D_XY=D_YX.$ See \cite{NomizuSasaki:/1994} for an introduction to affine geometry.
%The deviation of $\hat{\n}$ from the Levi-Civita connection of $\~g$ is measured by the {\it{cubic tensor}}
%\eq{C(X,Y,Z):=& -\fr 12(\hat{\n}_{X}\~g)(Y,Z)\\
%            =&\~g(D_{X}Y,Z)-\fr 12 \left[\overline{\mrm{{Rm}}}(X, Y)Z + \overline{\mrm{Rm}}(X, Z)Y\right]^{\dot{x}},
%        }
%where the superscript $\dot{x}$ denotes the $\dot{x}$ component in the splitting, \begin{align*}T N \simeq  x_{\ast}  (T M) \oplus \R \dot{x}.\end{align*}
%See \cite[Prop.~4.1]{NomizuSasaki:/1994} adjusted to the Riemannian setting.
\Theo{lemma}{AE-A}{
\eq{\label{AE-A-1} \dot{A}=A^{2}+uA+\~{\n}\grad_{\tilde{g}}u+D \grad_{\tilde{g}}u-\n \br{\fr 1f (\L^{t})^{\~\sharp}V}+\br{\overline{\mrm{Rm}}(\cdot,\dot{x})\dot{x}}^{\top}. }
%\eq{\label{AE-B} \dot{B}&=\~{\n}^{2}_{\~g}u+ C(\cdot,\grad_{\tilde{g}}u,\cdot)+\fr 1f \L(A(\cdot),\cdot)\\
%            &\hp{=}-\~g\br{\n\br{\fr 1f \L^{\~\sharp}V},\cdot}+\~g(\br{\overline{\mrm{Rm}}(\cdot,\dot{x})\dot{x}}^{\top},\cdot)}
}
\pf{Let $X,Y\in  T M.$ Differentiate \eqref{A} with respect to $\dot{x}$ to obtain
\begin{align*}\-\n_{\dot{x}}x_{\ast}(A(X))=-\-\n_{\dot{x}}\-\n_{X}\dot{x},\end{align*}
\begin{align*}[\dot{x},x_{\ast}(A(X))]+\-\n_{x_{\ast}(A(X))}\dot{x}=-\-\n_{X}\-\n_{\dot{x}}\dot{x}+\overline{\mrm{Rm}}(X,\dot{x})\dot{x}.\end{align*}
Thus using (\ref{BE-speed}) we get
\begin{align*}&x_{\ast}(\dot{A}(X))-x_{\ast}(A^{2}(X))-\overline{\mrm{Rm}}(X,\dot{x})\dot{x}\\
=&-\-\n_{X}\br{u\dot{x}-x_{\ast}(\grad_{\tilde{g}}u)+\fr 1f x_{\ast}((\L^{t})^{\~\sharp}V)}\\
            =&-\br{\-\n_{X}u}\dot{x}+ux_{\ast}(A(X))+x_{\ast}(\hat{\n}_{X}\grad_{\tilde{g}}u)+\~g(X,\grad_{\tilde{g}}u)\dot{x}\\
            &-\-\n_{X}\br{\fr 1f x_{\ast}((\L^{t})^{\~\sharp}V)}\\
            =&ux_{\ast}(A(X))+x_{\ast}(\~{\n}_{X}\grad_{\tilde{g}}u)+x_{\ast}(D_{X}\grad_{\tilde{g}}u)-\-\n_{X}\br{ x_{\ast}\br{\fr 1f(\L^{t})^{\~\sharp}V}}\\
            =&ux_{\ast}(A(X))+x_{\ast}(\~{\n}_{X}\grad_{\tilde{g}}u)+x_{\ast}(D_{X}\grad_{\tilde{g}}u)-x_{\ast}(\n_{X}\br{\fr 1f(\L^{t})^{\~\sharp}V})\\
            &+\s h(X,\fr 1f(\L^{t})^{\~\sharp}V)\nu.\end{align*}
        }
%``\eqref{AE-B}'': From \eqref{BE-tildeg} and \eqref{AE-A} we calculate
%\eq{\dot{B}(X,Y)&=\fr{d}{dt}\~g(A(X),Y)\\
%            &=-B(A(X),Y)-u\~g(A(X),Y)+\fr 1f \L(A(X),Y)\\
%            &\hp{=}+\~g(A^{2}(X),Y)+u\~g(A(X),Y)+\~g(\~{\n}_{X}\grad_{\tilde{g}}u,Y)\\
%            &\hp{=}+\~g(D_{X}\grad_{\tilde{g}}u,Y)-\~g\br{\n_{X}\br{\fr 1f \L^{\~\sharp}V},Y}\\
%            &\hp{=}+\~g(\br{\overline{\mrm{Rm}}(X,\dot{x})\dot{x}}^{\top},Y)\\
%            &=\~{\n}^{2}_{\~g}u(X,Y)+ C(X,\grad_{\tilde{g}}u,Y)+\fr 1f \L(A(X),Y)\\
%            &\hp{=}-\~g\br{\n_{X}\br{\fr 1f \L^{\~\sharp}V},Y}+\~g(\br{\overline{\mrm{Rm}}(X,\dot{x})\dot{x}}^{\top},Y)}
\Theo{lemma}{Ev-Lambda}{
\begin{align*}g(\partial_t\br{\fr{\L^{\sharp}}{f}}(X),Y)=&\fr 1f g(A(\L^{\sharp}(X)),Y)-\fr 1f \L(A(X),Y)\\
        &-\fr 1f \overline{\mrm{Rm}}(\grad_{\tilde{g}}u, X,\nu, Y)+\fr{1}{f^{2}}\overline{\mrm{Rm}}(\br{\L^{t}}^{\~\sharp}V,X,\nu,Y)\\
        &+\fr{1}{f}\-\n_{\dot{x}}\overline{\mrm{Rm}}(\dot{x},X,\nu,Y).\end{align*}
}
\pf{Differentiating the defining equation
\begin{align*}g(\fr{\L^{\sharp}}{f}(X),Y)=\fr 1f\L(X,Y)=\fr 1f\overline{\mrm{Rm}}(\dot{x},x_{\ast}X,\nu,x_{\ast}Y)\end{align*}
 with respect to $\dot{x}$ and using
\begin{align*}2A^{\flat}_{\mrm{sym}}(\L^{\sharp}(X),Y)&=g(A(\L^{\sharp}(X)),Y)+g(A(Y),\L^{\sharp}(X))\end{align*}
as well as \eqref{A}, \eqref{BE-g} and \eqref{BE-speed} yield the result.
}
\Theo{lemma}{Lambdasharptilde}{
\begin{align*}
g(\n_{Z}\br{\fr{(\L^{t})^{\~\sharp}V}{f}},Y)
=&-\n_{Z}\overline{\mrm{Rm}}(\dot{x},\cW^{-1}(Y),\nu,\dot{x})+\overline{\mrm{Rm}}(\dot{x},\cW^{-1}(Y),\dot{x},\cW(Z))\\
    &+\overline{\mrm{Rm}}(\dot{x},\cW^{-1}\br{(\n h(Z,\cW^{-1}(Y),\cdot))^{\sharp}},\nu,\dot{x})\\
    &+\overline{\mrm{Rm}}(\dot{x},\cW^{-1}((\overline{\mrm{Rm}}(\nu,\cW^{-1}(Y))Z)^{\top}),\nu,\dot{x})\\
    &+\s h(Z,\cW^{-1}(Y))\overline{\mrm{Rm}}(\dot{x},\nu,\nu,\dot{x})\\
    &-\L(A(Z),\cW^{-1}(Y))+2\L(\cW^{-1}(Y),A(Z)).
\end{align*}
}
\pf{Covariant differentiating the equation
\eq{\label{Lambdasharptilde-2}g(\fr{(\L^{t})^{\~\sharp}V}{f},Y)=\fr{h}{f}((\L^{t})^{\~\sharp}V,\cW^{-1}(Y))
=\L(\cW^{-1}(Y),V)=-\overline{\mrm{Rm}}(\dot{x},\cW^{-1}(Y),\nu,\dot{x})}
with respect to $Z$ gives
\eq{\nonumber\label{Lambdasharptilde-1}g(\n_{Z}\br{\fr{(\L^{t})^{\~\sharp}V}{f}},Y)
=&-g(\fr{(\L^{t})^{\~\sharp}V}{f},\n_{Z}Y)-\-\n_{Z}\overline{\mrm{Rm}}(\dot{x},\cW^{-1}(Y),\nu,\dot{x})\\
    &+\overline{\mrm{Rm}}(A(Z),\cW^{-1}(Y),\nu,\dot{x})-\overline{\mrm{Rm}}(\dot{x},\-\n_{Z}(\cW^{-1}(Y)),\nu,\dot{x})\\
            &-\overline{\mrm{Rm}}(\dot{x},\cW^{-1}(Y),\cW(Z),\dot{x})+\overline{\mrm{Rm}}(\dot{x},\cW^{-1}(Y),\nu,A(Z)).}
Moreover, by the Gaussian formula \eqref{Gauss},
\begin{align*}
\-\n_{Z}(\cW^{-1}(Y))=\n_{Z}(\cW^{-1}(Y))-\s h(Z,\cW^{-1}(Y))\nu.
\end{align*}
Putting this last relation as well as \eqref{Flow2} into \eqref{Lambdasharptilde-1} gives
\begin{align*}
g(\n_{Z}\br{\fr{(\L^{t})^{\~\sharp}V}{f}},Y)=&-g(\fr{(\L^{t})^{\~\sharp}V}{f},\n_{Z}Y)-\-\n_{Z}\overline{\mrm{Rm}}(\dot{x},\cW^{-1}(Y),\nu,\dot{x})\\
    &-\overline{\mrm{Rm}}(\dot{x},\n_{Z}(\cW^{-1}(Y)),\nu,\dot{x})+\s h(Z,\cW^{-1}(Y))\overline{\mrm{Rm}}(\dot{x},\nu,\nu,\dot{x})\\
        &+2\L(\cW^{-1}(Y),A(Z))-\L(A(Z),\cW^{-1}(Y))\\
    &+\overline{\mrm{Rm}}(\dot{x},\cW^{-1}(Y),\dot{x},\cW(Z)).
\end{align*}
To turn the fourth term on the right-hand side of this last identity to a tensorial term, we use the Codazzi equation
 \footnote{$g(\n_X \mc{W}(\mc{W}^{-1}(Y)),Z)=\n h(\mc{W}^{-1}(Y),Z,X)=\n h(\mc{W}^{-1}(Y),X,Z)+\overline{\mrm{Rm}}(\nu,\mc{W}^{-1}(Y),X,Z);$ therefore,
$\n_Z \mc{W}(\mc{W}^{-1}(Y))=(\n h(\mc{W}^{-1}(Y),Z,\cdot))^{\sharp}+(\overline{\mrm{Rm}}(\nu,\cW^{-1}(Y))Z)^{\top}.$}
\eq{\n_{Z}(\cW^{-1}(Y))=&-\cW^{-1}(\n_{Z}\cW(\cW^{-1}(Y)))+\cW^{-1}(\n_{Z}Y)\\
=&-\cW^{-1}\br{(\n h(Z,\cW^{-1}(Y),\cdot))^{\sharp}}+\cW^{-1}(\n_{Z}Y)\\
&-\cW^{-1}((\overline{\mrm{Rm}}(\nu,\cW^{-1}(Y))Z)^{\top})}
and
\eq{-g(\fr{(\L^{t})^{\~\sharp}V}{f},\n_{Z}Y)=\overline{\mrm{Rm}}(\dot{x},\cW^{-1}(\n_{Z}Y),\nu,\dot{x}).}
Therefore we arrive at
\begin{align*}
g(\n_{Z}\br{\fr{(\L^{t})^{\~\sharp}V}{f}},Y)
=&-\n_{Z}\overline{\mrm{Rm}}(\dot{x},\cW^{-1}(Y),\nu,\dot{x})\\
    &+\overline{\mrm{Rm}}(\dot{x},\cW^{-1}\br{(\n h(Z,\cW^{-1}(Y),\cdot))^{\sharp}},\nu,\dot{x})\\
    &+\overline{\mrm{Rm}}(\dot{x},\cW^{-1}((\overline{\mrm{Rm}}(\nu,\cW^{-1}(Y))Z)^{\top}),\nu,\dot{x})\\
    &+\s h(Z,\cW^{-1}(Y))\overline{\mrm{Rm}}(\dot{x},\nu,\nu,\dot{x})\\
    &-\L(A(Z),\cW^{-1}(Y))+2\L(\cW^{-1}(Y),A(Z))\\
        &+\overline{\mrm{Rm}}(\dot{x},\cW^{-1}(Y),\dot{x},\cW(Z)).
\end{align*}
}
We need one more lemma before calculating the main the evolution equation.
\Theo{lemma}{MixedTrace}{
% \eq{\mc{W}\circ A^2=\ad(A)\circ\L^{\sharp}+\L^{\sharp}\circ A+\ad(A)\circ\cW\circ A.}
\begin{align*}
\fr{1}{f}\operatorname{Tr}( f'\circ\cW\circ A^{2})&=\fr 1fd_{h}\zf\br{\L(\cdot,A(\cdot))}-\fr 1f d_{h}\zf\br{\L(A(\cdot),\cdot)}\\
&\hp{=}+\fr 1f d_{h}\zf\br{h(A(\cdot),A(\cdot))}.
\end{align*}
}
\pf{The claim follows from Proposition \ref{BilF} and Lemma \ref{SymmetriesB}:
\begin{align*}
\operatorname{Tr}( f'\circ\cW\circ A^{2})&=\operatorname{Tr}( f'\circ (h(\cdot,A^{2}(\cdot)))_{\mrm{sym}}^{\sharp_g})\\
&=d_{h}\zf\br{h(\cdot,A^{2}(\cdot))}\\
&=fd_{h}\zf\br{B(\cdot,A(\cdot))}\\        &=fd_{h}\zf\br{B(A(\cdot),\cdot)}+d_{h}\zf\br{\L(\cdot,A(\cdot))}-d_{h}\zf\br{\L(A(\cdot),\cdot)}\\                    &=d_{h}\zf\br{h(A(\cdot),A(\cdot))}+d_{h}\zf\br{\L(\cdot,A(\cdot))}-d_{h}\zf\br{\L(A(\cdot),\cdot)}.
\end{align*}
}

\Theo{lemma}{Ev-u-new}{Under the flow \eqref{Flow} we have
\eq{\label{Ev-u}\mc{L}u:=&\dot{u}-d_{h}\zf(\tilde{\nabla}^2u)-d_{h}\zf\br{\~g(D_{(\cdot)}\grad_{\~g}u,\cdot)}+\fr 1f d_{h}\zf\br{\overline{\mrm{Rm}}(\grad_{\tilde{g}}u,\cdot,\nu,\cdot)}\\
=&(\ln\p)'' \dot{s}^{2}+(\ln\p)'\ddot{s}+\frac{\dot{s}}{f}(\ln\p)' d_{\cW}f(\dot{\cW})+\fr 1f  d^2_{\cW}f(\dot{\cW},\dot{\cW})\\
    \hp{=}&+\fr 2f d_{h}\zf\br{h(A(\cdot),A(\cdot))}+\fr 2f \operatorname{Tr}( f'\circ (A\circ\L^{\sharp}-\L^{\sharp}\circ A))\\
    \hp{=}&+\s\br{1-\fr{d_{h}\zf\br{h(\cdot,\cdot)}}{f}}\overline{\mrm{Rm}}(\dot{x},\nu,\nu,\dot{x})+\fr 2f d_{h}\zf\br{\overline{\mrm{Rm}}(\cdot,\dot{x},\dot{x},\cW(\cdot))}\\
&+\fr 1fd_{h}\zf\br{\-\n\overline{\mrm{Rm}}(\dot{x},\cdot,\nu,\cdot,\dot{x})+\-\n\overline{\mrm{Rm}}(\dot{x},\cdot,\nu,\dot{x},\cdot)}.}
}
\pf{Note that
\begin{align*}
u=\fr{\dot{f}}{f}=(\ln \p)'\dot{s}+\fr 1f d_{\cW}f(\dot{\cW})=(\ln\p)'\dot{s}+\fr 1f  d_{\cW}f(A\circ\cW+\L^{\sharp}).
\end{align*}
Hence taking the time derivative,
\begin{align*}
\dot{u}=&(\ln\p)''\dot{s}^{2}+(\ln\p)'\ddot{s}+\fr 1f(\ln\p)'\dot{s} d_{\cW}f(\dot{\cW})\\
        &-\fr{u}{f} d_{\cW}f(A\circ\cW+\L^{\sharp})+\fr 1f  d_{\cW}^2f(\dot{\cW},\dot{\cW})+\fr 1f  d_{\cW}f(\dot{A}\circ\cW+A\circ\dot{\cW}+\del_{t}\L^{\sharp})\\
        =&(\ln\p)''\dot{s}^{2}+(\ln\p)'\ddot{s}+\fr 1f(\ln\p)'\dot{s} d_{\cW}f(\dot{\cW})+\fr 1f  d_{\cW}^2f(\dot{\cW},\dot{\cW})\\
        &+\fr 1f  d_{\cW}f(\dot{A}\circ\cW+A^{2}\circ\cW+A\circ\L^{\sharp}-uA\circ\cW)+ d_{\cW}f\br{\del_{t}\br{\fr{{\L^{\sharp}}}{f}}}.
\end{align*}
Since
$d_{\cW}f(T)=\operatorname{Tr}( f'\circ T)$
and $f'$ commutes with $\cW$, \cref{AE-A} implies that
\begin{align*}
\dot{u}=&(\ln\p)''\dot{s}^{2}+(\ln\p)'\ddot{s}+\fr 1f(\ln\p)'\dot{s}  d_{\cW}f(\dot{\cW})+\fr 1f   d_{\cW}^2f(\dot{\cW},\dot{\cW})\\
    &+\fr 1f\operatorname{Tr}( f'\circ\cW\circ(\dot{A}+A^{2}-uA))+ d_{\cW}f\br{\fr{A\circ\L^{\sharp}}{f}+\del_{t}\br{\fr{{\L^{\sharp}}}{f} }}\\
        =&(\ln\p)''\dot{s}^{2}+(\ln\p)'\ddot{s}+\fr 1f(\ln\p)'\dot{s}  d_{\cW}f(\dot{\cW})+\fr 1f   d_{\cW}^2f(\dot{\cW},\dot{\cW})\\
    &+\fr{2}{f}\operatorname{Tr}( f'\circ\cW\circ A^{2})+\fr 1f\operatorname{Tr}\br{ f'\circ\cW\circ\br{\~{\n}\grad_{\tilde{g}}u+D \grad_{\tilde{g}}u}}\\
    &-\fr 1f \operatorname{Tr}\br{ f'\circ\cW\circ\n\br{\fr 1f (\L^{t})^{\~\sharp}V}}+ d_{\cW}f\br{\fr{A\circ\L^{\sharp}}{f}+\del_{t}\br{\fr{{\L^{\sharp}}}{f} }}\\
    &+\fr 1f \operatorname{Tr}\br{ f'\circ\cW\circ\br{\overline{\mrm{Rm}}(\cdot,\dot{x})\dot{x}}^{\top}}.
\end{align*}
Rewriting the $\grad_{\tilde{g}}$-terms we obtain
\eq{\label{Ev-u-5-new}\dot{u}&-d_{h}\zf\br{\~{\n}^{2}u}-d_{h}\zf\br{\~g(D_{(\cdot)}\grad_{\~g}u,\cdot)}\\
    =&(\ln\p)''\dot{s}^{2}+(\ln\p)'\ddot{s}+\fr 1f(\ln\p)'\dot{s}  d_{\cW}f(\dot{\cW})+\fr 1f   d_{\cW}^2f(\dot{\cW},\dot{\cW})\\
    &+\fr 1f \operatorname{Tr}\br{ f'\circ\cW\circ\br{\overline{\mrm{Rm}}(\cdot,\dot{x})\dot{x}}^{\top}}+\fr{2}{f}\operatorname{Tr}( f'\circ\cW\circ A^{2})\\
    &-\fr 1f \operatorname{Tr}\br{ f'\circ\cW\circ\n\br{\fr 1f (\L^{t})^{\~\sharp}V}}+ d_{\cW}f\br{\fr{A\circ\L^{\sharp}}{f}+\del_{t}\br{\fr{{\L^{\sharp}}}{f} }}.}
Using the formulas from \cref{Ev-Lambda}, \cref{Lambdasharptilde} and \cref{MixedTrace}, we treat the last three terms of (\ref{Ev-u-5-new}) in order.
\item[(i)] From \cref{MixedTrace} we obtain
\eq{\label{Ev-u-4-new}\fr{2}{f}\operatorname{Tr}( f'\circ\cW\circ A^{2})=&\fr 2f d_{h}\zf\br{\L(\cdot,A(\cdot))}-\fr 2f\operatorname{Tr} (f'\circ\L^{\sharp}\circ A)+\fr 2f d_{h}\zf\br{h(A(\cdot),A(\cdot))},}
where we used Proposition \ref{BilF} to obtain
\begin{align*}
d_{h}\zf\br{\L(A(\cdot),\cdot)}=d_h\zf\br{g(\L^{\sharp}\circ A(\cdot),\cdot)}=\operatorname{Tr} (f'\circ\L^{\sharp}\circ A).
\end{align*}
\item[(ii)]
\cref{Lambdasharptilde} implies that\footnote{In an orthonormal frame $\{e_i\}$ we have
\begin{align*}
\sum_{i,j}\overline{\mrm{Rm}}(\dot{x},\cW^{-1}((\nabla h(e_i,f'(e_j),\cdot))^{\sharp}),\nu,\dot{x})&=\overline{\mrm{Rm}}(\dot{x},V,\nu,\dot{x})=-f\s\overline{\mrm{Rm}}(\dot{x},\nu,\nu,\dot{x}).
\end{align*}}
\eq{\label{Ev-u-2-new}&-\fr 1f \operatorname{Tr}\br{ f'\circ\cW\circ\n\br{\fr 1f (\L^{t})^{\~\sharp}V}}\\
 &=\fr{1}{f}d_{h}\zf(\nabla\overline{\mrm{Rm}}(\dot{x},\cdot,\nu,\dot{x},\cdot))-\fr{1}{f}d_{h}\zf(\overline{\mrm{Rm}}(\dot{x},\cdot,\dot{x},\cW(\cdot)))\\
    &\hp{=}+\s\overline{\mrm{Rm}}(\dot{x},\nu,\nu,\dot{x})
    +\frac{1}{f}d_{h}\zf(\Lambda(\cW^{-1}((\overline{\mrm{Rm}}(\nu,\cdot)(\cdot))^{\top}),V))\\
&\hp{=}-\frac{\s}{f}d_{h}\zf(h(\cdot,\cdot))\overline{\mrm{Rm}}(\dot{x},\nu,\nu,\dot{x})
+\frac{1}{f}d_{h}\zf(\L(A(\cdot),\cdot))-\frac{2}{f}d_{h}\zf(\L(\cdot,A(\cdot)))
.}
\item[(iii)] In view of \cref{Ev-Lambda} we have
\eq{\label{Ev-u-1-new}& d_{\cW}f\br{\fr{A\circ\L^{\sharp}}{f}+\del_{t}\br{\fr{{\L^{\sharp}}}{f} }}\\ =&\fr{2}{f}\operatorname{Tr} ( f'\circ A\circ{\L^{\sharp}})-\fr{1}{f}\operatorname{Tr}( f'\circ\L^{\sharp}\circ A)-\fr 1fd_{h}\zf\br{\overline{\mrm{Rm}}(\grad_{\tilde{g}}u,\cdot,\nu,\cdot)}\\
\hp{=}&+\fr{1}{f^{2}}d_{h}\zf\br{\overline{\mrm{Rm}}((\L^{t})^{\~\sharp}V,\cdot,\nu,\cdot)}+\fr 1f d_{h}\zf\br{\-\n\overline{\mrm{Rm}}(\dot{x},\cdot,\nu,\cdot,\dot{x})}.}
Also, note that $(\L^{t})^{\~\sharp}=f\mc{W}^{-1}\circ(\L^{t})^{\sharp}$; therefore,
\begin{align*}
\fr{1}{f^{2}}d_{h}\zf\br{\overline{\mrm{Rm}}((\L^{t})^{\~\sharp}V,\cdot,\nu,\cdot)}&=-\fr 1f d_{h}\zf\br{\overline{\mrm{Rm}}\br{\nu,\cdot,\cdot,\fr{(\L^{t})^{\~\sharp}V}{f}}}\\
        &=-\fr 1f d_{h}\zf \br{g(\fr{(\L^{t})^{\~\sharp}V}{f},(\overline{\mrm{Rm}}(\nu,\cdot)(\cdot))^{\top})}\\
        &=-\fr 1f d_{h}\zf\br{g(\mc{W}^{-1}\circ(\L^{t})^{\sharp}(V),(\overline{\mrm{Rm}}(\nu,\cdot)(\cdot))^{\top})}\\
        &=-\fr 1f d_{h}\zf\br{\L(\cW^{-1}((\overline{\mrm{Rm}}(\nu,\cdot)(\cdot))^{\top}),V)}.
\end{align*}
Putting these last three items all together gives
\eq{\label{Ev-u-6-new}&\fr{2}{f}\operatorname{Tr}( f'\circ\cW\circ A^{2})-\fr 1f \operatorname{Tr}( f'\circ\cW\circ\n\br{\fr 1f (\L^{t})^{\~\sharp}V})+d_{\cW}f(\fr{A\circ\L^{\sharp}}{f}+\del_{t}\br{\fr{{\L^{\sharp}}}{f} })\\
    =~& \fr 2f \operatorname{Tr}( f'\circ (A\circ \L^{\sharp}-\L^{\sharp}\circ A))-\fr 1f d_{h}\zf\br{\overline{\mrm{Rm}}(\grad_{\tilde{g}}u,\cdot,\nu,\cdot)}\\
    +&~\fr 1f d_{h}\zf\br{\-\n\overline{\mrm{Rm}}(\dot{x},\cdot,\nu,\cdot,\dot{x})}+\s\br{1-\fr{d_{h}\zf \br{h(\cdot,\cdot)}}{f}}\overline{\mrm{Rm}}(\dot{x},\nu,\nu,\dot{x})\\
    -&~\fr{1}{f}d_{h}\zf\br{\overline{\mrm{Rm}}(\dot{x},\cdot,\dot{x},\cW(\cdot))}+\fr 1f d_{h}\zf\br{\-\n\overline{\mrm{Rm}}(\dot{x},\cdot,\nu,\dot{x},\cdot)}+\fr 2f d_{h}\zf\br{h(A(\cdot),A(\cdot))}.}
Putting \eqref{Ev-u-6-new} into \eqref{Ev-u-5-new} gives the claimed result.
}
\section{Gauss Map and Duality}\label{Duality}
\label{gauss_duality}
In what follows, a semicolon denotes covariant derivatives with respect to the induced metric.
In this section, we give a brief review of a duality relation between strictly convex hypersurfaces of the unit sphere $\S^{n+1}$ and a duality relation between strictly convex hypersurfaces of the hyperbolic space with such of the de Sitter space. The relevant results can be found in \cite[Ch.~9, 10]{Gerhardt:/2006}. For convenience, we will state the main results here and stick to the notation in \cite{Gerhardt:/2006}.
\subsection*{Duality in the sphere} In this section, $\langle \cdot,\cdot\rangle$ denotes the inner product in $\mathbb{R}^{n+2}.$
Let $x\cn M_0\ra M\hra \S^{n+1}$ be a strictly convex closed hypersurface. Let the {\it{Gauss map}} $\~x\in T_x(\R^{n+2})$ represent the unit normal vector to $M$, $\nu\in T_x(\S^{n+1}).$ Then the mapping
\eq{\label{GaussMap}\~x\cn M_0\ra \S^{n+1}}
is also the embedding of a closed and strictly convex hypersurface. The geometry of $\~x$ is governed by the following theorem:
\Theo{thm}{SphereDuality}{\cite[Thm.~9.2.5]{Gerhardt:/2006}
Let $x\cn M_0\ra M\ra \S^{n+1}$ be a closed, connected, strictly convex hypersurface of class $C^{m},$ $m\geq 3,$ then the Gauss map $\~x$ in \eqref{GaussMap} is the embedding of a closed, connected, strictly convex hypersurface $\~M\sub \mathbb{S}^{n+1}$ of class $C^{m-1}.$ Viewing $\~M$ as a codimension $2$ submanifold in $\R^{n+2},$ its Gaussian formula is
\eq{\~x_{;ij}=-\~g_{ij}\~x-\~h_{ij}x,}
where $\~g_{ij},$ $\~h_{ij}$ are the metric and the second fundamental form of the hypersurface $\~M\sub \mathbb{S}^{n+1}$ and $x=x(\xi)$ is the embedding of $M$ which also represents the exterior normal vector of $\~M$. The second fundamental form $\~h_{ij}$ is defined with respect to the interior normal vector.

The second fundamental forms of $M,$ $\~M$ and the corresponding principal curvatures $\k_{i},$ $\~\k_{i}$ satisfy
\begin{align*}
h_{ij}=\~h_{ij}=\ip{\~x_{;i}}{x_{;j}},\quad \~\k_{i}=\k_{i}^{-1}.
\end{align*}
}
We point out that $\tilde{M}$ is called the polar set to $M$ and it has the following elegant representation:
$$\tilde{M}=\{y\in \mathbb{S}^{n+1}: \sup_{x\in M}\langle x,y\rangle=0\};$$
see \cite[Thm.~9.2.9]{Gerhardt:/2006}.

The following illustration shall give a clearer picture of the duality:

\includegraphics[width=0.95\textwidth]{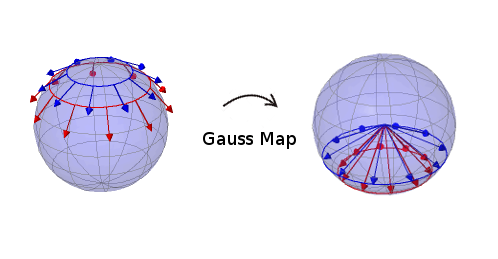}

\subsection*{Duality between hyperbolic space and de Sitter space} In this section, $\langle \cdot,\cdot\rangle$ denotes the inner product of $\mathbb{R}^{n+1,1}.$
The de Sitter space is the Lorentzian spaceform in the Minkowski space with constant sectional curvature $K_N=1:$
\begin{align*}\mathbb{S}^{n,1}=\{z\in\mathbb{R}^{n+1,1}: \langle z,z\rangle=1 \},\end{align*}
whereas the hyperbolic space is a Riemannian spaceform in the Minkowski space with constant sectional curvature $K_N=-1$:
\begin{align*}\mathbb{H}^{n+1}=\{z\in\mathbb{R}^{n+1,1}: \langle z,z\rangle=-1,~ z^{0}>0\},\end{align*}
where $z^0$ is the time coordinate.

Similarly, as for the sphere, given an embedding
$
x\cn M_0\ra M\sub\H^{n+1}
$
of a closed and strictly convex hypersurface, the representation $\~x\in T_x(\R^{n+1,1})$ of the exterior normal vector $\nu\in T_x(\H^{n+1})$ yields the embedding
\eq{\label{HypGaussMap}\~x\cn M_0\ra \~M\sub \S^{n,1}}
of a strictly convex, closed and spacelike hypersurface $\~M.$
We also call $\~x$ the {\it{Gauss map of $M$}} and similar to the spherical case we have the following theorem:
\Theo{thm}{DeSitterDuality}{\cite[Thm.~10.4.4]{Gerhardt:/2006}
Let $x\cn M\ra \H^{n+1}$ be a closed, connected, strictly convex hypersurface of class $C^{m},$ $m\geq 3,$ then the Gauss map $\~x$ as in \eqref{HypGaussMap} is the embedding of a closed, spacelike, achronal \footnote{This property is irrelevant for us.}, strictly convex hypersurface $\~M\sub \mathbb{S}^{n,1}$ of class $C^{m-1}.$ Viewing $\~M$ as a codimension $2$ submanifold in $\R^{n+1,1},$ its Gaussian formula is
\begin{align*}
\~x_{;ij}=-\~g_{ij}\~x+\~h_{ij}x,
\end{align*}
where $\~g_{ij},$ $\~h_{ij}$ are the metric and the second fundamental form of the hypersurface $\~M\sub \mathbb{S}^{n,1}$ and $x=x(\xi)$ is the embedding of $M$ which also represents the future directed normal vector of $\~M$. The second fundamental form $\~h_{ij}$ is defined with respect to the future directed normal vector, where the time orientation of $N$ is inherited from $\R^{n+1,1}$.

The second fundamental forms of $M,$ $\~M$ and the corresponding principal curvatures $\k_{i},$ $\~\k_{i}$ satisfy
\begin{align*}
h_{ij}=\~h_{ij}=\ip{\~x_{;i}}{x_{;j}},\quad \~\k_{i}=\k_{i}^{-1}.
\end{align*}
}
The hypersurface $\tilde{M}$ is called the polar set to $M$ and it can be represented as follows \cite[Thm.~10.4.8]{Gerhardt:/2006}:
$$\tilde{M}=\{y\in \mathbb{S}^{n,1}: \sup_{x\in M}\langle x,y\rangle=0\}.$$
In this model of the hyperbolic space the point $(1,0,\ldots,0)$ is called the {\it{Beltrami point}}. For a given strictly convex hypersurface $M\sub\H^{n+1}$, $M$ bounds a strictly convex body $\hat{M}$ of the hyperbolic space, cf.~\cite[Thm.~10.3.1]{Gerhardt:/2006}, and due to the homogeneity of the hyperbolic space, any point in $\hat{M}$ may act as Beltrami point after suitable ambient change of coordinates. Therefore, in addition to the statement of \cref{DeSitterDuality}, \cite[Thm.~10.4.9.]{Gerhardt:/2006} implies that the dual $\~M$ is contained in the future of the slice $\{z^0=0\}$,
\eq{\~M\sub \S^{n,1}_{+}=\{z\in \S^{n,1}\cn z^0>0\}.}
We will also need the reverse direction starting from a strictly convex, spacelike hypersurface in $\S^{n,1}.$
\Theo{thm}{HyperbolicDuality}{\cite[Thm.~10.4.5]{Gerhardt:/2006}
Let $x\cn \~M\ra \S^{n,1}$ be a closed, connected, spacelike, strictly convex hypersurface of class $C^{m},$ $m\geq 3,$ such that, when viewed as a codimension 2 submanifold in $\R^{n+1,1}$, its Gaussian formula is
\eq{\~x_{;ij}=-\~g_{ij}\~x+\~h_{ij}x,}
where $\~x=\~x(\xi)$ is the embedding, $x$ the future directed normal vector , and $\~g_{ij}$, $\~h_{ij}$ the induced metric and the second fundamental form of the hypersurface in $\S^{n,1}$. Then we define the Gauss map as $x=x(\xi)$
\eq{x\cn \~M\ra \H^{n+1}\sub\R^{n+1,1}.}
The Gauss map is the embedding of a closed, connected, strictly convex hypersurface $M$ in $\H^{n+1}.$
Let $g_{ij},$ $h_{ij}$ be the metric and the second fundamental form of $M$, then, when viewed as a codimension 2 submanifold, $M$ satisfies the relations
\begin{align*}
x_{ij}&=g_{ij}x-h_{ij}\~x,\\
h_{ij}=&\~h_{ij}=\ip{x_{;i}}{\~x_{;j}},\\
\~\k_{i}&=\k_{i}^{-1},
\end{align*}
where $\k_{i}$, $\~\k_{i}$ are the corresponding principal curvatures.
}
The following illustration shall give a clearer picture of the duality:

\includegraphics[width=0.95\textwidth]{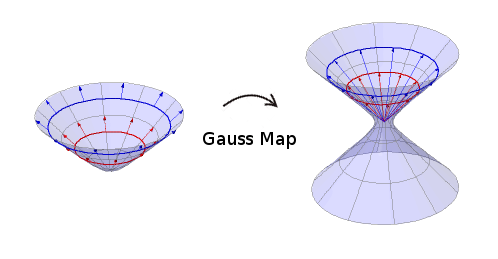}

\subsection*{Dual flows}
We want to deduce a duality relation for flows of strictly convex hypersurfaces in $\S^{n+1}$, as well as in $\H^{n+1}$ and $\S^{n,1}.$ A similar deduction of these results can be found in \cite[Sec.~4]{Gerhardt:/2015} and \cite{Yu:04/2016}.
For a curvature flow
\eq{\label{NormalFlow}\dot{x}=-\s f\nu,\quad \s=\ip{\nu}{\nu},}
we want to derive the flow equation of the Gauss maps $\~x$. Here $\ip{\cdot}{\cdot}$ represents the Euclidean and the Minkowski inner product respectively for flows in $\S^{n+1}$ and  $\H^{n+1}, \S^{n,1}$. In all three cases, the pair $x,\~x$ satisfies
\begin{align*}
\ip{x}{\~x}=0.
\end{align*}
Hence we have
\begin{align*}\ip{\dot{\~x}}{x}=-\ip{\~x}{\dot{x}}=-\ip{\~x}{- \s f\~x}= f.\end{align*}
Due to $\ip{\~x}{\~x_{;i}}=0$ and the Weingarten equation \cite[Lem.~9.2.4, Lem.~10.4.3]{Gerhardt:/2006},
\begin{align*}\ip{\dot{\~x}}{\~x_{;i}}=h^k_i\ip{\dot{\~x}}{x_{;k}}=-h^k_i\ip{\~x}{\dot{x}_{;k}}=h^k_if_{;k}.\end{align*}
Since $x=\~\nu$ and $\~x_{;i}$ span $T_{\~x}(\S^{n+1})$, $T_{\~x}(\S^{n,1})$ or $T_{\~x}(\H^{n+1})$ respectively, we obtain
\begin{align*}\dot{\~x}&=\ip{x}{x} f\~\nu+h^k_mf_{;k}\~g^{ml}\~x_{;l}\\
            &=\~\s f\~\nu+\~b^k_m\~g^{ml}f_{;k}\~x_{;l}\\
            &=\~\s f\~\nu+\~b^{kl}f_{;k}\~x_{;l},
\end{align*}
where $f$ is evaluated at $\mc{W}$. Let us put
\eq{\label{DualFlow-3}{\bf{f}}(\widetilde{\mc{W}}):=-f(\mc{W})=-\fr{1}{\~f(\mc{W}^{-1})}.}
Thus the flow of the polar hypersurfaces is governed by
\eq{\label{DualHyperbolic}\dot{\~x}=-\~\s{\bf{f}}\~\nu-\~b^{kl}{\bf{f}}_{;k}\~x_{;l},}
where $\~\s=\ip{\~\nu}{\~\nu}$ and ${\bf{f}}$ is evaluated at the ``correct" Weingarten map $\widetilde{\mc{W}}$. Hence we have shown a flow of the form
\eqref{NormalFlow} in the ambient spaces $\S^{n+1}, \H^{n+1},\S^{n,1}$ has a dual flow of the form \eqref{Flow} in the ambient spaces $\S^{n+1}, \S^{n,1},\H^{n+1}$ respectively.

\section{Locally symmetric spaces and proof of the main theorems}
In this section, we prove the Harnack inequalities. We restrict to locally symmetric spaces since in more general settings we do not know how to deal with the terms including derivatives of the Riemannian curvature tensor.

To prove our main theorems, we need a corollary of \cref{Ev-u-new} with bonus term $\beta$, cf. \cref{Homogeneous}. To prove \cref{Homogeneous}, we need to estimate $d_{\cW}^{2}f(\dot{\cW},\dot{\cW}).$ It is tempting to think that the mere convexity of the function $\Phi=\Phi(\k)$ (with the associated operator function $f$)   would be sufficient for this purpose. However, note that in \eqref{DerSymmFct-A}, the second term on the right-hand side requires the mixed terms $\dot{\cW}^{i}_{j}\dot{\cW}^{j}_{i}$ to be nonnegative while we are not aware whether $\dot{\cW}$ is $g$-selfadjoint in general, so some care should be taken. A similar issue arises when dealing with inverse concave curvature functions.
% The proof for the second statement, however appearing in a slightly different form, can be found in \cite[p.~112]{Urbas:/1991}. Here we need them in a somewhat different form.

\Theo{lemma}{d2f}{
 Let $N$ be a spaceform and $f$ satisfy \cref{SpeedAss}. If $F$ is convex, then we have
 \eq{d^{2}_{\cW}f(\dot{\cW},\dot{\cW})\geq \fr{p-1}{p}f^{-1}d_{\cW}f(\dot{\cW})^{2}.}
If $F$ is inverse concave, then  we have
\eq{d^{2}_{\cW}f(\dot{\cW},\dot{\cW})+2d_{\cW}f(\dot{\cW}\circ\cW^{-1}\circ\dot{\cW})\geq \fr{p+1}{p}f^{-1}d_{\cW}f(\dot{\cW})^{2}.}
This inequality is reversed if $F$ is inverse convex.
 }
\pf{Let $F$ be convex. We first the case consider $p=1$. In a spaceform we have
\eq{\label{Lambda}\L=\overline{\mrm{Rm}}(\dot{x},\cdot,\nu,\cdot)=K_{N}\br{\-g(\dot{x},\cdot)\-g(\nu,\cdot)-\-g(\dot{x},\nu)g}=K_{N}fg.} In view of \eqref{DerSymmFct-A}, there holds
\eq{d^{2}_{\cW}f(\dot{\cW},\dot{\cW})=d^{2}\Phi(\k)(\mrm{diag}(\dot{\cW}),\mrm{diag}(\dot{\cW}))+\sum_{i\neq j}^{n}\fr{\fr{\del\Phi}{\del\k_{i}}-\fr{\del\Phi}{\del\k_{j}}}{\k_{i}-\k_{j}}\dot{\cW}^{i}_{j}\dot{\cW}^{j}_{i},}
where $f(s,\nu,\cdot)$ is the associated operator function to $\Phi$ fibrewise. By \eqref{BE-W} we have
\eq{\dot{\cW}^{i}_{j}=A^{i}_{k}\cW^{k}_{j},\quad i\neq j.}
At a point $\xi\in M$ choose an orthonormal basis $\{\~e_{i}\}$ of $T_{\xi}M$ such that the $\~e_{i}$ are principal directions, i.e., in this basis we have
\eq{\label{g-ONB}g_{ij}=\d_{ij},\quad h_{ij}=\k_{i}\d_{ij},\quad \cW^{i}_{j}=\k_{i}\d^{i}_{j}.}
By scaling the coordinates
\eq{\label{h-ONB}e_{i}:=\fr{\~e_{i}}{\sqrt{h(\~e_{i},\~e_{i})}}=\fr{\~e_{i}}{\sqrt{\k_i}},}
we obtain the matrix representations
\eq{h=(h_{ij})=(\d_{ij}),\quad \cW=(\cW^{i}_{j})=(\k_{i}\d^{i}_{j}).}
Note that $B$ is symmetric (e.q., $A$ is $h$-selfadjoint); therefore, $A^{i}_{j}=A^{j}_{i}$. Thus for each pair $i\neq j$ we have
\eq{\dot{\cW}^{i}_{j}\dot{\cW}^{j}_{i}=A^{i}_{k}\cW^{k}_{j}A^{j}_{m}\cW^{m}_{i}=\k_{i}\k_{j}A^{i}_{j}A^{j}_{i}\geq 0.}
Convexity of $\Phi$ and \cite[Lemma~2.20]{Andrews:/1994b} yield $\fr{\fr{\del\Phi}{\del\k_{i}}-\fr{\del\Phi}{\del\k_{j}}}{\k_{i}-\k_{j}}\geq 0.$ Therefore, we arrive at
\eq{d^{2}_{\cW}f(\dot{\cW},\dot{\cW})\geq 0.}
For $p\neq 1$ we calculate
\eq{d_{\cW}f(\dot{\cW})=|p|\p\psi F^{p-1}d_{\cW}F(\dot{\cW})}
and
\eq{\label{d2f-2}d^{2}_{\cW}f(\dot{\cW},\dot{\cW})&=|p|\p\psi F^{p-1}d^{2}_{\cW}F(\dot{\cW},\dot{\cW})+|p|(p-1)\p\psi F^{p-2}d_{\cW}F(\dot{\cW})^{2}\\
                            &\geq \fr{p-1}{|p|\p\psi F^{p}}d_{\cW}f(\dot{\cW})^{2}=\fr{p-1}{pf}d_{\cW}f(\dot{\cW})^{2}.}
Now suppose that $F$ is inverse concave. Again we first consider the case $p=1$.
For the inverse symmetric function $\~\Phi$ the corresponding $\~F$ has the property that
\eq{\label{InvConc-1}\~F(\cW)=\fr{1}{F(\cW^{-1})}}
and similarly for $f$ and $\~f$.
So for all $B\in T^{1,1}(M)$ we get
\eq{d_{\cW}\~f(B)=\~f^{2}d_{\cW^{-1}}f(\cW^{-1}\circ B\circ \cW^{-1})}
and
\eq{\label{d2f-1}d_{\cW}^{2}\~f(B,B)&=2\~f^{3}\br{d_{\cW^{-1}}f(\cW^{-1}\circ B\circ \cW^{-1})}^{2}\\
            &\hp{=}-\~f^{2}d^{2}_{\cW^{-1}}f(\cW^{-1}\circ B\circ \cW^{-1},\cW^{-1}\circ B\circ \cW^{-1})\\
            &\hp{=}-2\~f^{2}d_{\cW^{-1}}f(\cW^{-1}\circ B\circ \cW^{-1}\circ B\circ \cW^{-1}),}
where $\~f=\~f(\cW)$ and $f=f(\cW^{-1})$.

Take $B=\cW\circ\dot{\cW}\circ \cW.$
As above $d^{2}_{\cW}\~f(B,B)$ is given explicitly by
\eq{d^{2}_{\cW}\~f(B,B)=d^{2}\~\Phi(\k)(\mrm{diag}(B),\mrm{diag}(B))+\sum_{i\neq j}^{n}\fr{\fr{\del\~\Phi}{\del\k_{i}}-\fr{\del\~\Phi}{\del\k_{j}}}{\k_{i}-\k_{j}}B^{i}_{j}B^{j}_{i},}
where, in the same basis $\{e_{i}\}$ as above,
\eq{B^{i}_{j}=\cW^{i}_{k}A^{k}_{m}\cW^{m}_{l}\cW^{l}_{j}=\k_{i}\k_{j}^{2}A^{i}_{j}.}
Hence $B^{i}_{j}B^{j}_{i}\geq 0$ for each pair of $i\neq j$. So due to the concavity of $\~\Phi$ we have
\eq{d^{2}_{\cW}\~f(B,B)\leq 0.}
From \eqref{d2f-1} we obtain
\eq{d^{2}_{\cW^{-1}}f(\dot{\cW},\dot{\cW})+2d_{\cW^{-1}}f(\dot{\cW}\circ\cW\circ\dot{\cW})\geq 2\~fd_{\cW^{-1}}f(\dot{\cW})^{2}.}
This proves the claim when $p=1$.
For the general case, we use \eqref{d2f-2} to obtain
\eq{d^{2}_{\cW}f(\dot{\cW},\dot{\cW})&\geq |p|\p\psi F^{p-1}\br{2F^{-1}d_{\cW}F(\dot{\cW})^{2}-2d_{\cW}F(\dot{\cW}\circ\cW^{-1}\circ\dot{\cW})}\\
                    &\hp{=}+\fr{p-1}{pf}d_{\cW}f(\dot{\cW})^{2}\\
                    &=-2d_{\cW}f(\dot{\cW}\circ\cW^{-1}\circ\dot{\cW})+\fr{p+1}{pf}d_{\cW}f(\dot{\cW})^{2}.}
}
\Theo{lemma}{Homogeneous}{
Let the ambient space $N$ be locally symmetric (i.e.,
$\-{\n}\overline{\mrm{Rm}}=0$).
Suppose $f$ satisfies \cref{SpeedAss}. For $p\neq 0,-1$ and $\beta\in \R$, we put
\begin{align*}q:=t(u-\beta)+\fr{p}{p+1}.\end{align*}
Then for $p>0$ and any strictly convex solution to \eqref{Flow} there holds
\eq{\label{Ev-q}\mc{L}q&\geq\fr{t}{\p^{2}}\br{\p''\p+\fr{(1-p)\p'^{2}}{p}}f^{2}+\fr{2t\s}{p}\fr{\p'}{\p}fu+\fr{p+1}{p}(u-\beta)q\\
            &\hp{=}-t\fr{p+1}{p}(u-\beta)^{2}+t\fr{p-1}{p}u^{2}+\fr{2t}{p}\br{u-\fr{d_{\cW}f(\L^{\sharp})}{f}}^{2}\\
            &\hp{=}+t\s\br{1-\fr{d_{h}\zf\br{h(\cdot,\cdot)}}{f}}\overline{\mrm{Rm}}(\dot{x},\nu,\nu,\dot{x})+\fr{2t}{f} d_{h}\zf\br{\overline{\mrm{Rm}}(\cdot,\dot{x},\dot{x},\cW(\cdot))},
}
If $\overline{\mrm{Rm}}=0$ and $p<-1$, then this inequality still holds. If $\overline{\mrm{Rm}}=0$ and $-1<p<0,$ then the inequality is reversed.
}
\pf{
We consider two cases.

{\bf{Case 1:}}  $N=\R^{n+1}$ or $N=\R^{n,1}.$
In this case, equation \eqref{Ev-u} reads
\begin{align*}
\mathcal{L}u=&\br{\ln\p}''\dot{s}^2+\br{\ln\p}'\ddot{s}+\frac{\dot{s}}{f}\br{\ln\p}'d_{\cW}f(\dot{\cW})+\fr 1f d_{\cW}^{2}f(\dot{\cW},\dot{\cW})\\
        &+\fr 2f d_{h}\zf\br{h(A(\cdot),A(\cdot))}.
\end{align*}
Recall that $s=\s\ip{x}{\nu}.$
Hence taking derivative with respect to time yields
\begin{align*}\dot{s}&=\s\ip{\dot{x}}{\nu}=-\s f,\\
\ddot{s}&=-\s\dot{f}=\fr{\p'}{\p}f^2-\s d_{\cW}f(\dot{\mathcal{W}})=\fr{\p'}{\p}f^2-\s d_{\cW}f(A\circ\cW).
\end{align*}
We also have
\eq{\label{R=0-3} u=-\s \fr{\p'}{\p}f+\fr 1f d_{\cW}f(\dot{\mathcal{W}})=-\s \fr{\p'}{\p}f+\fr 1f d_{\cW}f(A\circ\cW).}
Moreover, since $f'$ and $\mc{W}$ commute we have
\eq{d_{h}\zf(h(A(\cdot),A(\cdot)))=\mrm{Tr}(f'\circ\cW\circ A^2)=\mrm{Tr}(f'\circ A^2\circ\cW)=d_{\cW}f(\dot{\cW}\circ\cW^{-1}\circ\dot{\cW}).}
These identities in conjunction with \cref{d2f} implies that if $p(p+1)>0$,
\eq{\label{R=0-1}
\mathcal{L}u&\geq \br{\ln\p}''\dot{s}^2+\br{\ln\p}'\ddot{s}+\frac{\dot{s}}{f}\br{\ln\p}'d_{\cW}f(\dot{\cW})+\fr{p+1}{p}f^{-2}d_{\cW}f(\dot{\mathcal{W}})^{2}\\
                &=\fr{\p''}{\p}f^{2}-\fr{2\s\p'}{\p}d_{\cW}f(\dot{\cW})+\fr{p+1}{p}u^{2}+\fr{2(p+1)\s\p'}{p\p}fu+\fr{p+1}{p}\fr{\p'^{2}}{\p^{2}}f^{2}\\
                &=\fr{p+1}{p}u^{2}+\fr{1}{\p^{2}}\br{\p''\p+\fr{(1-p)\p'^{2}}{p}}f^{2}+\fr{2\s}{p}\fr{\p'}{\p}fu,}
and if $-1<p<0$, the inequality is reversed.

{\bf{Case 2:}} $N$ has nonzero curvature and $p>0$. In case $N$ is a spaceform or as well in the case $f=\psi H$, due to \eqref{Lambda}, we have
\begin{align*}\mrm{Tr}(f'\circ (A\circ\L^{\sharp}-\L^{\sharp}\circ A))=0.\end{align*}
By \cref{InvConc} (or Remark \ref{prop2.4rem}) we have
\begin{align*}d_{h}\zf\br{h(A(\cdot),A(\cdot))}&=d_h\zf(g(A(\cdot),\cW\circ A(\cdot)))\\
					&=d_h\zf(g(\cdot,\ad(A)\circ\cW\circ A(\cdot)))\\
					&=d_{\cW}f(\ad(A)\circ\cW\circ A)\\
                        &=d_{\cW}f(\ad(\cW\circ A)\circ \cW^{-1}\circ\cW\circ A )\\
                        &\geq \fr 1p f^{-1}d_{\cW}f(\cW\circ A)^{2}\\
                        &=\fr{1}{p}f^{-1}d_{\cW}f(\dot{\cW}-\L^{\sharp})^{2}\\
                        &=\frac{1}{pf}\br{d_{\cW}f(\dot{\cW})^{2}-2d_{\cW}f(\dot{\cW})d_{\cW}f(\L^{\sharp})+d_{\cW}f(\L^{\sharp})^{2}}.
                        \end{align*}
Also, \cref{d2f} gives
\eq{d_{\cW}^{2}f(\dot{\cW},\dot{\cW})\geq \fr{p-1}{p}f^{-1}d_{\cW}f(\dot{\cW})^{2}.}
Note that this last inequality still holds if $F=\psi H$.

Using these observations and that $fu=d_{\cW}f(\dot{\mathcal{W}})$, we arrive at
\begin{align*}&\fr 2f d_{h}\zf\br{h(A(\cdot),A(\cdot))}+\fr 1fd_{\cW}^{2}f(\dot{\cW},\dot{\cW})\\
    \geq &~\fr{p+1}{p}f^{-2}d_{\cW}f(\dot{\cW})^2-\fr 4pf^{-2} d_{\cW}f(\dot{\cW})d_{\cW}f(\L^{\sharp})+\fr 2p f^{-2}d_{\cW}f(\L^{\sharp})^{2}\\
    =&~\fr{p+1}{p}u^2-\frac{4}{p}\frac{d_{\cW}f(\L^{\sharp})}{f}u+\frac{2}{p}\left(\frac{d_{\cW}f(\L^{\sharp})}{f}\right)^2.
\end{align*}

Therefore, from \eqref{Ev-u} we deduce that if $p>0$, then
\eq{\label{LocHom-1}\mathcal{L}u&\geq \frac{p+1}{p}u^2-\frac{4}{p}\frac{d_{\cW}f(\L^{\sharp})}{f}u+\frac{2}{p}\left(\frac{d_{\cW}f(\L^{\sharp})}{f}\right)^2
\\
&\hp{=}+\s\br{1-\fr{d_{h}\zf\br{h(\cdot,\cdot)}}{f}}\overline{\mrm{Rm}}(\dot{x},\nu,\nu,\dot{x})+\fr 2f d_{h}\zf\br{\overline{\mrm{Rm}}(\cdot,\dot{x},\dot{x},\cW(\cdot))}.
        }
In both cases (1) and (2), using \eqref{R=0-1} and \eqref{LocHom-1}, we obtain
\begin{align*}\mc{L}q&=u-\beta+t\mc{L} u\\
            &=\fr{p+1}{p}(u-\beta)q-t\fr{p+1}{p}(u-\beta)^{2}+t\mc{L}u\\
            &\geq\fr{t}{\p^{2}}\br{\p''\p+\fr{(1-p)\p'^{2}}{p}}f^{2}+\fr{2t\s}{p}\fr{\p'}{\p}fu+\fr{p+1}{p}(u-\beta)q\\
            &\hp{=}-t\fr{p+1}{p}(u-\beta)^{2}+t\fr{p-1}{p}u^{2}+\fr{2t}{p}\br{u-\fr{d_{\cW}f(\L^{\sharp})}{f}}^{2}\\
            &\hp{=}+t\s\br{1-\fr{d_{h}\zf\br{h(\cdot,\cdot)}}{f}}\overline{\mrm{Rm}}(\dot{x},\nu,\nu,\dot{x})+\fr{2t}{f} d_{h}\zf\br{\overline{\mrm{Rm}}(\cdot,\dot{x},\dot{x},\cW(\cdot))},
\end{align*}
with reversed inequality if $-1<p<0$ and $\overline{\mrm{Rm}}=0.$
}
We are now ready to prove various Harnack inequalities.
\subsection*{Euclidean and Minkowski space}
In the case that the ambient curvature vanishes, we obtain the following Harnack inequalities for anisotropic flows claimed in Theorem \ref{Euclidean}. In particular, the theorem includes and extends the well-known Harnack inequalities from \cite{Andrews:09/1994} in the Euclidean space and they are completely new in the Minkowski space.
\Theo{thm}{R=0}{
Let $N$ be either the Euclidean or the Minkowski space and let the assumptions of \cref{Euclidean} be satisfied. Then along \eqref{Flow}, if $p>0$ or $p<-1$, there holds
\begin{align*}tu+\fr{p}{p+1}\geq 0,\end{align*}
and if $-1<p<0$ the inequality is reversed. Moreover, if $p=-1,\varphi=1$ and $F$ is inverse concave, then
$\inf u$
is increasing. Also, if $p=-1,\varphi=1$ and $F$ is inverse convex, then
$\sup u$
is decreasing.
In particular, \cref{Euclidean} holds.
}
\pf{Apply \eqref{Ev-q} with $\beta=0$ to obtain that $q$ satisfies
\begin{align*}\mc{L}q\geq \fr{2\s}{p}\fr{\p'}{\p}fq-\fr{2\s}{p+1}\fr{\p'}{\p}f+\fr{p+1}{p}uq\end{align*}
with reversed inequality if $-1<p<0$. For the Euclidean ambient space, the maximum principle gives the Harnack estimate. If $N=\mathbb{R}^{n,1}$, due to our assumptions in Theorem \ref{Euclidean}, we can apply the maximum principle on the compact set $K$ and prove the claimed Harnack inequalities in each case. This proves \cref{Euclidean} (and also \cref{QuotientsofEuc}) in case $p\neq -1$.

If $p=-1,\varphi=1$, note that in view of Lemma \ref{d2f}, the right-hand side of \eqref{Ev-u} is positive (negative) if $F$ is inverse concave (inverse convex).
}
\subsection*{Locally symmetric Einstein spaces of non-negative sectional curvature}
Here we obtain a Harnack inequality for the mean curvature flow:
\Theo{thm}{LocSymm}{
Suppose $N$ is a Riemannian locally symmetric Einstein space with non-negative sectional curvature. Let $f=\psi(\nu)H$ with $\psi\in C^{\8}(\mathbb{U})$ invariant under parallel transport. Then for any strictly convex solution to \eqref{Flow} there holds
\eq{\label{Bonus}t\br{u-\fr{\-R}{n+1}}+\fr{1}{2}\geq 0,}
where $\-R$ is the scalar curvature. In particular, \cref{Einstein} holds.
}
\pf{
We use \eqref{Ev-q} with
$\beta=\fr{\-R}{n+1},$
where $R$ is the scalar curvature. In this situation we have $\s=1$ and $d_{\cW}f(\cW)=H$ and hence the last line of \eqref{Ev-q} is non-negative. Furthermore, there holds
\begin{align*}d_{\cW}f(\L^{\sharp})=-\mrm{Tr}(\overline{\mrm{Rm}}(\cdot,\dot{x},\nu,\cdot))=-\overline{\mrm{Rc}}(\dot{x},\nu)=\fr{\-R}{n+1}f.\end{align*}
Hence the claim follows from the maximum principle applied to \eqref{Ev-q}.
}
\subsection*{Riemannian spaces of constant positive curvature}
For the spherical space, inequality \eqref{Bonus} is  the Harnack inequality with a bonus term in \cite{BryanIvaki:08/2015}.
The next theorem recovers our  Harnack inequalities without bonus terms in \cite{BIS1}.
\Theo{thm}{Sphere}{
Let $N$ be a Riemannian spaceform of sectional curvature $K_{N}=1$ and $f$ satisfy \cref{SpeedAss} with $0<p\leq 1$. Then for any strictly convex solution to \eqref{Flow} there holds
\begin{align*}tu+\fr{p}{p+1}\geq 0.\end{align*}
In particular, \cref{SphereHarnack}-(1) holds.
}
\pf{
The last line of \eqref{Ev-q} is non-negative. Using \eqref{Lambda}
we calculate
\eq{-t\fr{p+1}{p}u^{2}&+t\fr{p-1}{p}u^{2}+\fr{2t}{p}\br{u-\fr{d_{\cW}f(\L^{\sharp})}{f}}^{2}\\
&=-\fr{4t}{p}\fr{d_{\cW}f(\L^{\sharp})}{f}u+\fr{2t}{p}\br{\fr{d_{\cW}f(\L^{\sharp})}{f}}^{2}\\
&\geq -\fr{4}{p}\fr{d_{\cW}f(\L^{\sharp})}{f}q+\fr{4}{p+1}\fr{d_{\cW}f(\L^{\sharp})}{f}.
    }
Hence the maximum principle implies the claim.
}
By applying the dual flow method developed in \Cref{Duality}, we obtain pseudo-Harnack inequalities
for a class of inverse curvature flows.
 \Theo{thm}{ExpandingSphere}{
Suppose $N=\S^{n+1}$ and $F$ is a positive, strictly monotone, inverse convex and $1$-homogeneous curvature function. Let $-1\leq p<0$ and $f=-F^{p}.$ Then for any strictly convex solution of \eqref{NormalFlow} we have
\begin{align*}\del_t\br{ft^{\fr{p}{p-1}}}\leq 0.\end{align*}
In particular, \cref{pseudoHarnack}-(1) holds.
}
\pf{
The dual flow of \eqref{DualHyperbolic} with speed
\begin{align*}{\bf{f}}(\~{\mc{W}})=-f(\mc{W})=-\mrm{sgn}(p)F^{p}(\mc{W})=\mrm{sgn}(-p)\~F^{-p}(\~{\mc{W}})\end{align*}
satisfies the assumptions of \cref{Sphere}, which in particular implies
\begin{align*}\del_t\br{{\bf{f}} t^{\fr{-p}{-p+1}}}\geq 0.\end{align*}
}
\subsection*{Lorentzian spaces of constant positive curvature}
For flows of spacelike hypersurfaces in Lorentzian manifolds of nonvanishing curvature, the second line of \eqref{Ev-q} can behave rather differently, since $\nu$ is timelike. In the de Sitter space of constant sectional curvature $K_{N}=1$, we obtain a similar result as in the spherical case, but only for flows with principal curvatures bounded by $1$. This further assumption is equivalent to convexity by horospheres for the dual hypersurfaces in the hyperbolic space and hence seems to be a natural assumption for flows in the de Sitter space.
\Theo{thm}{DeSitter}{
Let $N$ be a Lorentzian spaceform of sectional curvature $K_{N}=1$ and let $f$ satisfy \cref{SpeedAss} with $0<p\leq 1.$ Then for any spacelike solution $x$ of \eqref{Flow} that the condition $0<\kappa_i\leq 1$ is always satisfied on $[0,T^{\ast})$ there holds
\begin{align*}tu+\fr{p}{p+1}\geq 0.\end{align*}
In particular, \cref{SphereHarnack}-(2) holds.
}
\pf{
Recall that $-V=\dot{x}+\sigma f\nu$ and it is a spacelike vector. We start with the following observations:
\begin{align*}
\overline{\mrm{Rm}}(\dot{x},\nu,\nu,\dot{x})&=-\-g(\dot{x},\dot{x})-\-g(\dot{x},\nu)^2\\
&=f^2-\-g(V,V)-f^2\leq 0,
\end{align*}
 and
\begin{align*}
d_{h}\zf\br{\overline{\mrm{Rm}}(\cdot,\dot{x},\dot{x},\cW(\cdot))}&=d_{h}\zf\br{\-g(\dot{x},\dot{x})g(\cdot,\cW(\cdot))-\-g(\cdot,\dot{x})\-g(\dot{x},\cW(\cdot))}\\
&=\mrm{Tr}_{(f'\circ\cW)^{\sharp}}\br{-f^{2}g+g(V,V)g-g(\cdot,V)g(\cdot,V)}\\
&\geq -f^{2}d_{\cW}f(\cW)=-pf^3.
\end{align*}
The identity \eqref{Ev-q} with $\beta=0$ and these last two inequalities
imply that
\eq{\label{xx}\mathcal{L}q\geq&\br{\fr{p+1}{p}u-\fr{4}{p}\fr{d_{\cW}f(\L^{\sharp})}{f}}q
+\fr{2t}{p}\left(\br{\fr{d_{\cW}f(\L^{\sharp})}{f}}^2-(pf)^{2}\right)\\
            &+\fr{4}{p+1}\fr{d_{\cW}f(\L^{\sharp})}{f}.}
In addition, since $0<\kappa_i\leq 1$ and $\L^{\sharp}=f\operatorname{Id}$, the monotonicity of $f$ gives
\[pf=d_{\cW}f(\cW)\leq d_{\cW}f(\operatorname{Id})=\frac{d_{\cW}f(\L^{\sharp})}{f}.\]
The result now follows from the maximum principle.
}
\Theo{rem}{}{
In the de Sitter space, we cannot expect to obtain a Harnack estimate with a bonus term for mean curvature flow as in the spherical case. To see that, we will look at ancient solutions with $0<\kappa_i\leq 1$ to the mean curvature flow.

The evolution equation of $H$ is given by
\[\partial_tH=\Delta H+T\ast\nabla H-|A|^2H+nH.\]
If there was a Harnack inequality for mean curvature flow of the following form
\[\partial_tH-nH+\frac{H}{2t}\geq 0,\]
then for an ancient solution we would have
$\partial_tH-nH\geq 0.$
So evolution equation of $H$ would yield
$\Delta H+T\ast\nabla H-|A|^2H\geq 0;$ therefore, $H(\cdot,t)=0.$
}
%\Theo{thm}{MCF with bonus}{Let $N=\mathbb{S}^{n,1}$ and $f=H$. Then under \eqref{Flow} we have
%\[\partial_tH-nH+\frac{H}{2t}\geq 0.\]}
%\pf{Define $q:=w-nt$. Using (\ref{xx}) and ignoring its last two terms (since their sum is non-negative), we deduce that for $f=H:$
%\[\dot{q}\geq \Delta {q}+T\ast\n q+2\left(u-n\right)q.\]
%Hence by the maximum principle, $q$ remains positive. This implies that
%\[\partial_t H-nH+\frac{H}{2t}\geq 0.\]
%}

With precisely the same proof as for \cref{Sphere}, we obtain, using \eqref{DualHyperbolic} and \cref{DeSitter}, the following pseudo-Harnack inequality for expanding flows of the hyperbolic space, which is to our knowledge the first such inequality for hypersurface flows in the hyperbolic space:
\Theo{thm}{Hyperbolic}{Let $N=\H^{n+1}$ and $F$ be a positive, strictly monotone, inverse convex and 1-homogeneous curvature function. If $-1\leq p<0$, then any horoconvex solution to \eqref{NormalFlow} with speed $f=-F^{p}$
satisfies
\[\del_t\br{ft^{\fr{p}{p-1}}}\leq 0.\]
In particular, \cref{pseudoHarnack}-(2) holds.
}
\pf{The speed of the dual flow \eqref{DualHyperbolic} is
\[{\bf{f}}(\~{\mc{W}})=-f(\mc{W})=-\mrm{sgn}(p)F^p(\mc{W})=\mrm{sgn}(-p)\~F^{-p}(\~{\mc{W}}).\]
Thus the assumptions of \cref{DeSitter} are satisfied with $0<-p\leq 1$; therefore,
\[\del_t\br{{\bf{f}} t^{\fr{-p}{-p+1}}}\geq 0\]
and the claim follows.
}

\section{Cross curvature flow}\label{cross}
Let $(M^3,g)$ be a Riemannian 3-manifold with negative sectional curvature. The cross curvature tensor is defined by
\[c_{ij}:=(E^{-1})_{ij}\det E=\frac{1}{2}g_{ik}g_{jl}\mu^{kpq}\mu^{lrs}E_{pr}E_{qs}=\frac{1}{8}\mu^{pqk}\mu^{rsl}R_{ilpq}R_{kjrs},\]
where $E_{ij}:=R_{ij}-\frac{1}{2}Rg_{ij}$ is the Einstein tensor, $R_{ijkl}$ is the Riemann curvature tensor, $R_{ij}$ is the Ricci curvature tensor, $R$ is the scalar curvature, $\det E:=\det E_{ij}/\det g_{ij}$ and $\mu^{ijk}$ are the components of the volume form.

A one-parameter family of 3-manifolds $(M,g(t))$ with negative sectional curvature is a solution of the XCF if
\[\partial_tg_{ij}=2c_{ij}.\]
Now suppose the metrics are locally isometrically embeddable in Minkowski space $\mathbb{R}^{3,1}$. The following observation is due to Andrews, which recently appeared in \cite{AndrewsChenFangMcCoy:/2015}.

Recall that the Gauss equation in $\mathbb{R}^{3,1}$ reads\footnote{To provide a better comparability with the references mentioned in this section, the convention for the Riemannian curvature tensor here differs from our convention in the previous sections.}
\[
R_{ijkl} = -(h_{ik}h_{jl} - h_{il}h_{jk}).
\]
Tracing with respect to $g^{ik}$ gives
\[
R_{jl} = -(Hh_{jl} - h^k_lh_{jk}), \quad R = -(H^2 - |A|^2),
\]
where $|A|^2=g^{ik}g^{jl}h_{ij}h_{kl}.$
Thus we have
\[
E_{ij} = \left(\frac{H}{2}g_{ij} - h_{ij}\right)H + \left(h^k_ih_{kj} - \frac{1}{2}|A|^2g_{ij}\right).
\]
In an orthonormal frame which diagonalizes the second fundamental form, we get for $i=1$:
\[
\begin{split}
E_{11} &= \frac{1}{2}\left(H^2 - |A|^2\right) + h^1_1 h_{11} - Hh_{11} \\
&=h_{11}h_{22} + h_{11}h_{33} + h_{22}h_{33} + h_{11}^2 - \left(h_{11} + h_{22} + h_{33}\right) h_{11} \\
&= h_{22}h_{33},
\end{split}
\]
and similarly for $i=2,3$. That is,
\[
E = \begin{pmatrix}
\kappa_2 \kappa_3 & 0 & 0 \\
0 & \kappa_1\kappa_3 & 0 \\
0 & 0 & \kappa_1\kappa_2
\end{pmatrix},
\]
where $\kappa_i$ denote the principal curvatures. In particular, $\det E = K^2,$ where $K$ is the Gauss curvature. If $M$ is strictly convex, then the matrix $E$ is positive definite. In this case, the cross curvature tensor is
\[
c_{ij} = ({\det} E) (E^{-1})_{ij} = \begin{pmatrix}
\kappa_1^2 \kappa_2 \kappa_3 & 0 & 0 \\
0 & \kappa_1\kappa_2^2 \kappa_3 & 0 \\
0 & 0 & \kappa_1\kappa_2\kappa_3^2
\end{pmatrix} = Kh_{ij}.
\]
Now the uniqueness result of Buckland \cite{Buckland:/2006} shows that $(M,g(t))$ is a solution of (\ref{FlowStandard}) with $N=\mathbb{R}^{3,1}$, $f=K$.

The Harnack inequality for the cross curvature flow for metrics that are locally isometrically embeddable in Minkowski space $\mathbb{R}^{3,1}$ now follows from \Cref{Euclidean}:
\begin{align*}\label{cross harnack}
\partial_t \sqrt{{\det} E} - \frac{1}{\sqrt{{\det}E}}E^{ij} \nabla_i \left(\sqrt{{\det}E}\right) \nabla_j \left(\sqrt{{\det}E}\right) + \frac{3}{4 t} \sqrt{{\det}E} \geq 0.
\end{align*}

\bibliographystyle{amsplain}
\bibliography{Bibliography}
\end{document}

%% file: StandardPaper2.tex
\usepackage[colorlinks=true,linkcolor=blue,citecolor=blue,urlcolor=blue,hypertexnames=false]{hyperref}
\usepackage{bookmark}
\usepackage{amsthm,thmtools,amssymb,amsmath,amscd}

\usepackage{fancyhdr}
\usepackage{esint}
\usepackage{enumerate}

\usepackage{pictexwd,dcpic}

\usepackage{stmaryrd}

\declaretheorem[name=Theorem,numberwithin=section]{thm}
\declaretheorem[name=Remark,style=remark,sibling=thm]{rem}

\declaretheorem[name=Proposition,sibling=thm]{prop}
\declaretheorem[name=Definition,style=definition,sibling=thm]{defn}

\declaretheorem[name=Example,style=remark,sibling=thm]{example}

\declaretheorem[name=Remark,numbered=no]{no-rem}

\numberwithin{equation}{section}

\usepackage{cleveref}
\crefname{lemma}{Lemma}{Lemmata}
\crefname{prop}{Proposition}{Propositions}
\crefname{thm}{Theorem}{Theorems}
\crefname{cor}{Corollary}{Corollaries}
\crefname{defn}{Definition}{Definitions}
\crefname{example}{Example}{Examples}
\crefname{rem}{Remark}{Remarks}
\crefname{ass}{Assumption}{Assumptions}
\crefname{notation}{Notation}{Notation}

\usepackage{autonum}
\makeatletter
\newcommand{\restore@Environment}[1]{%
  \AtBeginDocument{%
    \csletcs{#1*}{#1}%
    \csletcs{end#1*}{end#1}%
  }%
}
\forcsvlist\restore@Environment{alignat,equation,gather,multline,flalign,align}
\makeatother

\renewcommand{\~}{\tilde}
\renewcommand{\-}{\bar}

\newcommand{\cn}{\colon}
\newcommand{\sub}{\subset}

\newcommand{\Z}{\mathbb{Z}}

\newcommand{\R}{\mathbb{R}}

\renewcommand{\S}{\mathbb{S}}
\renewcommand{\H}{\mathbb{H}}

\newcommand{\8}{\infty}

\renewcommand{\a}{\alpha}

\renewcommand{\d}{\delta}

\renewcommand{\k}{\kappa}
\renewcommand{\l}{\lambda}

\newcommand{\s}{\sigma}
\newcommand{\p}{\varphi}

\renewcommand{\O}{\Omega}

\newcommand{\G}{\Gamma}

\renewcommand{\L}{\Lambda}

\newcommand{\cL}{\mathcal{L}}

\newcommand{\cW}{\mathcal{W}}

\newcommand{\cF}{\mathcal{F}}

\newcommand{\del}{\partial}

\newcommand{\n}{\nabla}

\newcommand{\ip}[2]{\left\langle #1,#2 \right\rangle}
\newcommand{\fr}[2]{\frac{#1}{#2}}
\newcommand{\x}{\times}

\DeclareMathOperator{\pr}{pr}

\DeclareMathOperator{\tr}{Tr}

\DeclareMathOperator{\grad}{grad}

\newcommand{\Theo}[3]{\begin{#1}\label{#2} #3 \end{#1}}
\newcommand{\pf}[1]{\begin{proof} #1 \end{proof}}
\newcommand{\eq}[1]{\begin{equation}\begin{alignedat}{2} #1 \end{alignedat}\end{equation}}

\newcommand{\br}[1]{\left(#1\right)}

\newcommand{\ra}{\rightarrow}
\newcommand{\hra}{\hookrightarrow}

\newcommand{\mc}{\mathcal}
\renewcommand{\it}{\textit}
\newcommand{\mrm}{\mathrm}

\newcommand{\hp}{\hphantom}

\newcommand*\xbar[1]{%
  \hbox{%
    \vbox{%
      \hrule height 0.2pt % The actual bar
      \kern0.5ex%         % Distance between bar and symbol
      \hbox{%
        \kern-0.1em%      % Shortening on the left side
        \ensuremath{#1}%
        \kern-0.1em%      % Shortening on the right side
      }%
    }%
  }%
}